\documentclass[preprint]{imsart}
\RequirePackage[OT1]{fontenc}
\RequirePackage{amsthm,amsmath,amssymb,nicefrac,bm,upgreek,verbatim,mathtools}
\RequirePackage[numbers,sort&compress]{natbib}
\RequirePackage[colorlinks,citecolor=mblue,urlcolor=mblue]{hyperref}

\RequirePackage{dsfont}
\RequirePackage[mathscr]{eucal}

\RequirePackage{enumitem}
\RequirePackage[normalem]{ulem}

\RequirePackage[margin=1in]{geometry}

\RequirePackage{color}

\startlocaldefs

\definecolor{mred}{rgb}{.7,.0,.0}
\definecolor{dred}{rgb}{.5,.0,.0}
\definecolor{mblue}{rgb}{.0,.0,.8}
\definecolor{dblue}{rgb}{.0,.0,.5}
\definecolor{mgreen}{rgb}{.0,.6,.0}
\definecolor{dgreen}{rgb}{.0,.4,.0}
\definecolor{dmagenta}{rgb}{.4,.1,.5}
\newcommand{\ftn}[1]{\footnote{\bfseries\color{mred}#1}}
\renewcommand{\ftn}[1]{\relax}

\newcommand{\sW}{{\mathscr{W}}}

\numberwithin{equation}{section}

 \numberwithin{dummy}{section}
\theoremstyle{plain}
\newtheorem{thm}{Theorem}[section]
\newtheorem{lem}{Lemma}[section]

\newtheorem{proposition}{Proposition}[section]

\theoremstyle{definition}

\newtheorem{definition}{Definition}[section]

\theoremstyle{remark}
\newtheorem{remark}{Remark}[section]

\newcommand{\stkout}[1]{\ifmmode\text{\sout{\ensuremath{#1}}}\else\sout{#1}\fi}

\usepackage[capitalize,nameinlink]{cleveref}

\crefname{section}{Section}{Sections}
\crefname{subsection}{Subsection}{Subsections}
\crefname{notation}{Notation}{Notations}
\crefname{hypothesis}{Hypothesis}{Conditions}
\crefname{assumption}{Assumption}{Assumptions}
\crefname{lemma}{Lemma}{Lemmas}
\crefname{lem}{Lemma}{Lemmas}
\crefname{cor}{Corollary}{Corollaries}
\crefname{thm}{Theorem}{Theorems}
\crefname{claim}{Claim}{Claims}

\Crefname{figure}{Figure}{Figures}

\crefformat{equation}{\textup{#2(#1)#3}}
\crefrangeformat{equation}{\textup{#3(#1)#4--#5(#2)#6}}
\crefmultiformat{equation}{\textup{#2(#1)#3}}{ and \textup{#2(#1)#3}}
{, \textup{#2(#1)#3}}{, and \textup{#2(#1)#3}}
\crefrangemultiformat{equation}{\textup{#3(#1)#4--#5(#2)#6}}%
{ and \textup{#3(#1)#4--#5(#2)#6}}{, \textup{#3(#1)#4--#5(#2)#6}}%
{, and \textup{#3(#1)#4--#5(#2)#6}}

\Crefformat{equation}{#2Equation~\textup{(#1)}#3}
\Crefrangeformat{equation}{Equations~\textup{#3(#1)#4--#5(#2)#6}}
\Crefmultiformat{equation}{Equations~\textup{#2(#1)#3}}{ and \textup{#2(#1)#3}}
{, \textup{#2(#1)#3}}{, and \textup{#2(#1)#3}}
\Crefrangemultiformat{equation}{Equations~\textup{#3(#1)#4--#5(#2)#6}}%
{ and \textup{#3(#1)#4--#5(#2)#6}}{, \textup{#3(#1)#4--#5(#2)#6}}%
{, and \textup{#3(#1)#4--#5(#2)#6}}

\crefdefaultlabelformat{#2\textup{#1}#3}

\newcommand{\process}[1]{{\{#1(t)\}_{t\ge0}}}
\newcommand{\ttup}[1]{\textup{(}#1\textup{)}}
\newcommand{\Ind}{\mathds{1}} 
\newcommand{\tUsm}{\widetilde{\mathfrak{U}}_{\mathrm{sm}}} 

\newcommand{\Exp}{{\mathbb{E}}} 
\newcommand{\Prob}{{\mathbb{P}}} 
\newcommand{\RR}{{\mathbb R}} 

\newcommand{\NN}{{\mathbb N}} 
\newcommand{\Z}{\mathbb{Z}} 
\newcommand{\D}{\mathrm{d}} 
\newcommand{\E}{\mathrm{e}} 
\newcommand{\df}{\coloneqq} 
\newcommand{\Id}{\mathbb{I}}
\newcommand{\sB}{{\mathscr{B}}} 
\newcommand{\cB}{{\mathcal{B}}} 
\newcommand{\cM}{{\mathcal{M}}} 
\newcommand{\Usm}{\mathfrak{U}_{\mathrm{SM}}} 
\newcommand{\X}{\mathbb{X}}

\newcommand{\Lg}{{\mathcal{L}}}

\newcommand{\Lyap}{{\mathscr{V}}}

\newcommand{\cP}{{\mathcal{P}}} 

\newcommand{\T}{{\mathbb{T}}} 
\newcommand{\abs}[1]{\lvert#1\rvert}
\newcommand{\norm}[1]{\lVert#1\rVert}
\newcommand{\babs}[1]{\bigl\lvert#1\bigr\rvert}

\newcommand{\bnorm}[1]{\bigl\lVert#1\bigr\rVert}

\DeclareMathOperator*{\diag}{diag}
\DeclareMathOperator{\trace}{Tr}

\DeclareMathOperator{\Lip}{Lip}

\endlocaldefs
\begin{document}
\begin{frontmatter}

\title{Subexponential upper and lower bounds in Wasserstein distance 
for Markov processes}
\runtitle{Subexponential upper and lower bounds in Wasserstein Distance}

\begin{aug}
\author{\fnms{Nikola} \snm{Sandri{\'c}}\thanksref{c}\ead[label=e1]{nsandric@math.hr}},
\author{\fnms{Ari} \snm{Arapostathis}\thanksref{a}\ead[label=e2]{ari@utexas.edu}}
\and
\author{\fnms{Guodong} \snm{Pang}\thanksref{b}\ead[label=e3]{gdpang@rice.edu}}

\runauthor{N.~Sandri{\'c},  A.~Arapostathis,  and G.~Pang}

\affiliation{University of Zagreb\thanksmark{m1}, University of Texas at Austin\thanksmark{m2},\\ and
Pennsylvania State University\thanksmark{m3}}

\address[c]{Department of Mathematics,
	University of Zagreb\\ Bijeni\v{c}ka cesta 30,
	10000 Zagreb, Croatia\\
	\printead{e1}}

\address[a]{Department of Electrical and Computer Engineering\\
University of Texas at Austin,
2501 Speedway, EER 7.824,
Austin, TX~78712\\
\printead{e2}}

\address[b]{Department of Computational and Applied Mathematics\\
Rice University,
Houston, TX~77005\\
\printead{e3}}

\end{aug}

\begin{abstract}
In this article, relying on  Foster-Lyapunov drift conditions, we establish
subexponential upper and lower bounds on the rate of convergence in the
$\mathrm{L}^p$-Wasserstein distance for a
class of irreducible and aperiodic Markov processes.
We further discuss these results in the context of Markov L\'evy-type processes.
In the lack of irreducibility and/or aperiodicity properties, we obtain
exponential ergodicity in the  $\mathrm{L}^p$-Wasserstein distance for a class of
It\^{o} processes under an asymptotic flatness (uniform dissipativity) assumption.
Lastly,  applications of these results to specific processes are presented,
including Langevin tempered diffusion processes, piecewise Ornstein--Uhlenbeck
processes with jumps under constant and stationary Markov controls, and backward
recurrence time chains, for which  we provide a sharp characterization of the
rate of convergence via matching upper and lower bounds. 
\end{abstract}



\end{frontmatter}


\section{Introduction}\label{S1}

One of the classical directions in the analysis of Markov processes centers
around their ergodic properties. In this article, we focus on both
qualitative and quantitative aspects of this problem.
Let $\X$ be a locally compact Polish space, i.e.
a  locally compact separable completely metrizable topological space.
Denote the corresponding metric by $\mathsf{d}$, and let $\T=\RR_+$ or $\Z_+$ be
the time parameter set. We endow $(\X,\mathsf{d})$ with its Borel $\sigma$-algebra
$\mathfrak{B}(\X)$. Further, let
$\bigl(\Omega,\mathcal{F}, \{\mathcal{F}_t\}_{t\in\T},\{\theta_t\}_{t\in\T},
\{X(t)\}_{t\in\T},\{\Prob_x\}_{x\in\X}\bigr)$, denoted by $\{X(t)\}_{t\in\T}$
in the sequel, be a time-homogeneous conservative strong Markov process with
c\`adl\`ag sample paths (when $\T=\RR_+$) and state space
$\bigl(\X,\mathfrak{B}(\X)\bigr)$, in the sense of \cite{BG-68}. 
Here, $(\Omega,\mathcal{F}, \Prob_x)_{x\in\X}$ is a family of probability spaces and $\{X(t)\}_{t\in\T}$ satisfies $\Prob_x(X(0)=x)=1$, $\{\mathcal{F}_t\}_{t\in\T}$ is a filtration on $(\Omega,\mathcal{F})$ (non-decreasing family of sub-$\sigma$-algebras of $\mathcal{F}$)  and $\{\theta_t\}_{t\in\T}$  is a family of shift operators on $\Omega$ satisfying $X(t)\circ \theta_s = X(t+s)$ for all $s, t\in\T$. Recall, $\{X(t)\}_{t\in\T}$ is said to be conservative if $\Prob_x(X(t)\in\X)=1$ for all $t\in\T$ and $x\in\X$.
In the present article, we present (sharp) sufficient conditions under which
$\{X(t)\}_{t\in\T}$ admits a unique invariant probability measure $\uppi(\D x)$,
and which ensure that the marginals of $\{X(t)\}_{t\in\T}$ converge to $\uppi(\D x)$,
as $t\to\infty$, in the $\mathrm{L}^{p}$-Wasserstein distance at 
exponential and subexponential rates.

\subsection{Summary of the results}

Before stating the main results of this article, we introduce some notation we need
in the sequel.
Denote by $p(t,x,\D y)\df\Prob_x(X(t)\in\D y)$ for $t\in\T$ and $x\in\X$, the
transition kernel of $\{X(t)\}_{t\in\T}$.
We endow $\T$ with the standard (Euclidean Borel in the case when $\T=\RR_+$, and
discrete when $\T=\Z_+$) $\sigma$-algebra. The process $\{X(t)\}_{t\in\T}$ is called
\begin{enumerate}
	\item [\ttup{i}]
	irreducible if there exists a $\sigma$-finite measure $\upvarphi(\D x)$ on
	$\mathfrak{B}(\X)$ such that whenever $\upvarphi(B)>0$ we have
	$\int_{\T}p(t,x,B)\,\uptau(\D t)>0$ for all $x\in\X$, where $\uptau(\D t)$ stands
	for the Lebesgue measure on $\T$ when $\T=\RR_+$, and the counting measure
	when $\T=\Z_+$;
	\item [\ttup{ii}]
	transient if it is irreducible, and there exist
	$\{b_k\}_{k\in\NN}\subset[0,\infty)$ and
	a covering $\{B_k\}_{k\in\NN}\subseteq\mathfrak{B}(\X)$ of $\X$,
	such that
	$\int_{\T}p(t,x,B_k)\,\uptau(\D{t})\le b_k$ for all $x\in\X$ and $k\in\NN$;
	\item [\ttup{iii}]
	recurrent if it is irreducible, and $\upvarphi(B)>0$ implies that
	$\int_{\T}p(t,x,B)\,\uptau(\D{t})=\infty$ for all $x\in\X$;
	\item [\ttup{iv}]
	aperiodic if there exists $t_0>0$ such that $\{X_{kt_0}\}_{k\in\Z_+}$ is irreducible,
	in the case when $\T=\RR_+$; and
	there does not exist a partition $\{B_1,\dots,B_k\}\subseteq\mathfrak{B}(\X)$ with
	$k\ge2$ of $\X$ such that $p(1,x,B_{i+1}) = 1$ for all $x\in B_i$ and all
	$1 \le i \le k -1$, and $p(1,x, B_1) = 1$ for all $x \in B_k$,
	in the case when $\T=\Z_+$.\end{enumerate}
Let us remark that if $\{X(t)\}_{t\in\T}$ is irreducible, 
then it is either transient or recurrent (see \cite[Theorem~2.3]{Tweedie-1994}). 
A Borel measure $\uppi(\D x)$ on $\X$ is called invariant for
$\{X(t)\}_{t\in\T}$ if
$\int_\X p(t,x,\D y)\, \uppi(\D x)=\uppi(\D y)$ for all $t\in\T$.
It is well known that if $\{X(t)\}_{t\in\T}$ is
recurrent,  then it possesses a unique (up to constant multiples) invariant measure 
(see \cite[Theorem~2.6]{Tweedie-1994}).
If the invariant measure is
finite, then it may be normalized to a probability measure. If
$\{X(t)\}_{t\in\T}$ is recurrent with finite invariant measure, then it is called 
positive recurrent; otherwise it is called null recurrent. Note that a transient 
Markov process cannot have a finite invariant measure.
A set $C\in\mathfrak{B}(\X)$ is called petite for $\{X(t)\}_{t\in\T}$ if there
exist a probability measure  $\upchi(\D t)$ on $\T$ and a non-trivial Borel
measure $\upnu_\upchi(\D x)$ on $\X$, such that
\begin{equation*}\int_{\T}p(t,x,B)\,\upchi(\D t)\,\ge\,\upnu_\upchi(B)\end{equation*} for all $x\in C$ and $ B\in\mathfrak{B}(\X).$
Recall that petite sets play a role of singletons for Markov processes on general
state spaces (see \cite[Chapter~5]{Meyn-Tweedie-Book-2009} for a detailed discussion). 
Denote by $\cP(\X)$ the class of all Borel probability measures on $\X$, and for
$f\in\cB(\X)$ (the space of real-valued Borel measurable functions on $\X$) let
$\cP_f(\X)$ denote the class of all $\upmu\in\cP(\X)$ with the property that
$\int_{\X}\abs{f(x)}\,\upmu(\D{x})<\infty$.
When $f(x)=\bigl(\mathsf{d}(x_0,x)\bigr)^p$ for some $p>0$ and $x_0\in\X$,
we denote this as $\cP_p(\X)$ .
We adopt the usual notation
\begin{equation*}
	\upmu P_t(\D y)=\int_{\X}p(t,x,\D y)\,\upmu(\D{x})\,,\qquad\text{and}\qquad
	\upmu\bigl(f\bigr)=\int_{\X}f(x)\,\upmu(\D{x})
\end{equation*}
for $t\in\T$, $x\in\X$, $\upmu\in\cP(\X)$ and $f\in\cB(\X)$.
Therefore, with $\updelta_x$ denoting the Dirac measure concentrated
at $x\in\X$, we have $\updelta_x P_t(\D y) = p(t,x,\D y)$.
Finally, recall that the
$\mathrm{L}^p$-Wasserstein distance on $\cP_p(\X)$ with $p\ge1$ is defined by
\begin{equation*}
	\sW_p(\upmu_1,\upmu_2)\,\df\,\inf_{\Pi\in\mathcal{C}(\upmu_1,\upmu_2)}
	\biggl( \int_{\X\times\X}\bigl(\mathsf{d}(x,y)\bigr)^{p}\,
	\Pi(\D{x},\D{y})\biggr)^{\nicefrac{1}{p}}\,,
\end{equation*}
where $\mathcal{C}(\upmu_1,\upmu_2)$ is the family of couplings of
$\upmu_1(\D x)$ and $\upmu_2(\D x)$,
i.e. $\Pi\in\mathcal{C}(\upmu_1,\upmu_2)$ if, and only if, $\Pi(\D x, \D y)$
is a probability
measure on $\X\times\X$ having $\upmu_1(\D x)$ and $\upmu_2(\D x)$ as its marginals.
It is well known that $\mathcal{P}_p(\X)$ is a 
complete separable metric space under the metric $\sW_p$
\cite[Theorem~6.18]{Villani-Book-2009}.
The topology generated by $\sW_p$ on $\cP_p(\X)$ is finer than the
Prokhorov topology, i.e. the topology of weak convergence.

We now state the main results of this article.

\begin{thm}\label{T1.1}
	Suppose that $\{X(t)\}_{t\in\T}$ is irreducible and aperiodic, and
	there exist a continuous $\Lyap\colon\X\to[1,\infty)$,
	a constant $b>0$, a nondecreasing differentiable concave function
	$\phi\colon[1,\infty)\to(0,\infty)$,
	and a (topologically) closed petite set $C\subseteq\X$ such that 
	\begin{equation}\label{ET1.1A}
	\Exp_x\bigl[\Lyap(X(t))\bigr] - \Lyap(x) \,\le\,
	b\int_{[0,t)}\Exp_x\bigl[\Ind_C(X(s))\bigr]
	\uptau(\D{s}) - \int_{[0,t)} \Exp_x\bigl[\phi\circ\Lyap(X(s))\bigr]\uptau(\D{s})
	\end{equation}
for all $(t,x)\in T\times\X$.
	Assume further that $\sup_{x\in C}\Lyap(x)<\infty$, and
	\begin{equation}\label{ET1.1B}
	c\,\df\,\inf_{x\in\X}\,\dfrac{\phi\circ\Lyap(x)}
	{\bigl(1+\mathsf{d}(x,x_0)\bigr)^{\eta}} \,>\, 0
	\end{equation} 
	for some $\eta\ge1$ and some (and therefore any) $x_0\in\X$.
	Then $\{X(t)\}_{t\in\T}$ admits a unique invariant
	$\uppi\in\mathcal{P}_{\phi\circ\Lyap}(\X)$.
	In addition, with $\Phi(t)\df \int_1^t\frac{\D{s}}{\phi(s)}$ and
	$r(t) \df \phi\circ\Phi^{-1}(t)$, the following hold. 
	\begin{enumerate}
		\item [\ttup i]
		If $\displaystyle\lim_{t\to\infty}\phi'(t)=0$, then for some $\Bar c>0$
		we have
		\begin{align}
			\left(1\vee\bigl(r(t)\bigr)^{\nicefrac{(\eta-1)}{\eta}}\right)\,
			\sW_1\bigl(\updelta_x P_t,\uppi\bigr) &\,\le\,
			\Bar c\, \Lyap(x) \qquad \forall\,(t,x)\in \T\times\X\,,\label{ET1.1C}
			\intertext{and}
			\int_\T\left(1\vee\bigl(r(t)\bigr)^{\nicefrac{(\eta-1)}{\eta}}\right)\,
			\sW_1\bigl(\updelta_x P_t,\updelta_yP_t\bigr)\,\uptau(\D{t}) &\,\le\,
			\Bar c\,\bigl(\Lyap(x)+\Lyap(y)\bigr)\quad\forall\, x,y\in\X\,.\label{ET1.1D}
		\end{align}

		\item [\ttup{ii}]
		If $\displaystyle\lim_{t\to\infty}\phi'(t)=0$,
		then for any $p\in[1,\eta]$ there exists $\Tilde c>0$ such that 
		\begin{equation}\label{ET1.1E}
		\left(1\vee \left(t^{\nicefrac{(\eta-p)}{p}}\wedge t^{\nicefrac{(1-p)}{p}}\right)
		\bigl(r(t)\bigr)^{\nicefrac{(\eta-1)}{p\eta}}\right)\,
		\sW_p(\updelta_x P_t,\uppi)
		\,\le\, \Tilde c\, \bigl(\Lyap(x) + \overline{m}_{\eta}\bigr)
		\end{equation}
for all $(t,x)\in \T\times\X$,
where $\overline m_\eta=\uppi\bigl(\bigl(\mathsf{d}(x_0,\cdot\,)\bigr)^\eta\bigr)$.

\smallskip		
		\item[\ttup{iii}]
		If $\phi(t)=\hat c\,t$
		for some  $\hat c>0$, then there exist  $\check c>0$ and $\gamma>0$, such that 
		\begin{equation}\label{ET1.1F}
		\E^{\gamma t}\,\sW_1\bigl(\updelta_x P_t,\uppi\bigr)\,\le\,
		\check c\, \Lyap(x)\qquad \forall\,(t,x)\in \T\times\X\,.
		\end{equation}
		In addition, for any $p\in[1,\eta]$ there
		exists $\breve{c}>0$ such that
		\begin{equation}\label{ET1.1G}
		\bigl(1\vee t^{\nicefrac{\eta}{p}-1}\bigr)\,
		\sW_p(\updelta_x P_t,\uppi)
		\,\le\, \breve c\,\bigl(
		\Lyap(x)+\overline{m}_{\eta}\bigr)^{\nicefrac{1}{p}}
		\qquad \forall\,(t,x)\in \T\times\X\,.
		\end{equation}
	\end{enumerate}
\end{thm}

The results in \cref{T1.1} should be compared to equations (2.3) and 
(2.5) in \cite[Theorems~2.1 and 2.4]{Butkovsky-14}
(see also \cite[Theorem~3\,\ttup{ii}]{DFM-16}  and \cite[Chapter 4]{Kulik-Book-2018}).
The underlying metric $\mathsf{d}$ is assumed to be bounded in \cite{Butkovsky-14}.
The starting point is a Foster-Lyapunov condition of the form
in \cref{ET1.1A}, and the irreducibility and aperiodicity assumptions
are replaced by a closely related structural property: the metric $\mathsf{d}$
is contracting, and the sublevel sets of $(x,y)\mapsto\Lyap(x)+\Lyap(y)$ are
$\mathsf{d}$-small (see (3) and (4) in \cite[Theorems~2.1 and 2.4]{Butkovsky-14}).
Then an analogous estimate to \cref{ET1.1C} holds for the corresponding
$\sW_1$-distance. Observe that when $\mathsf{d}$ is bounded,
the relation in \cref{ET1.1A} trivially holds for any $\eta\ge0$.
Provided $\{X(t)\}_{t\in\T}$ is irreducible and aperiodic,
this gives an analogous result to the one obtained in
\cite[Theorems~2.1 and 2.4]{Butkovsky-14} (in $\sW_1$-distance) without
assuming either contraction properties of $\mathsf{d}$ or $\mathsf{d}$-smallness
of the sublevel sets of $(x,y)\mapsto\Lyap(x)+\Lyap(y)$. 
The proof of \cref{T1.1} relies on \cite[Theorem~3.2]{Douc-Fort-Guilin-2009}
and \cite[Theorem~2.8]{DFMS-04}, where, under the assumptions of \cref{T1.1},
the authors show ergodicity
of $\{X(t)\}_{t\in\T}$ in the $f$-norm with rate $\Psi_1\circ r(t)$
and $f(x)=\Psi_2\circ \phi\circ\Lyap(x)\vee1$, for any pair
$(\Psi_1^{-1},\Psi_2^{-1})$ of Young's functions.
Recall, 	
 for a signed Borel measure $\upmu(\D x)$ on $\X$ and a function
$f\colon\X\to[1,\infty]$   the so-called $f$-norm of $\upmu(\D x)$ is defined as
\begin{equation}\label{EFNORM}\norm{\upmu}_f\,\df\,\sup_{g\in\cB(\X),\, \abs{g}\le f}\,
\babs{\upmu(g)}\,,\end{equation} generalizing the usual total variation norm $\norm{\upmu}_{\mathrm{TV}}\df\sup_{g\in\cB(\X),\, \abs{g}\le 1}\,
\babs{\upmu(g)}$.
We remark here that convergence in the $f$-norm does not in general imply
convergence in the $\sW_p$-distance, and \textit{vice versa} (see \cref{S3}
for examples of such Markov processes).

In the following theorem we establish a lower bound for $\sW_p$-convergence,
which matches the upper bounds obtained in \cref{ET1.1C,ET1.1E}.
For $\gamma\in C([0,1];\X)$ (the space of continuous mappings from $[0,1]$ to $\X$)
let
\begin{equation*}\Lambda(\gamma)\,\df\,
\sup_{k\in\NN}\,\sup_{0=u_0<u_1<\dotsb<u_{k-1}
	<u_k=1}\Big(\textsf{d}\bigl(\gamma(u_0),\gamma(u_1)\bigr)+\cdots+\textsf{d}
\bigl(\gamma(u_{k-1}),\gamma(u_k)\bigr)\Big)\,.\end{equation*}
The space $\X$ is called a length space if
\begin{equation*}\textsf{d}(x,y)\,=\,\inf_{\gamma\in C([0,1];\X)}
\left\{ \Lambda(\gamma)\colon \gamma(0)=x,\ \gamma(1)=y \right\}
\qquad \forall\, x,y\in\X\,.\end{equation*}

\begin{thm}\label{T1.2}
	Assume that $\X$ is a length space, $\{X(t)\}_{t\in\T}$ satisfies \cref{ET1.1A},
	and there exist
	a Lipschitz continuous function $L\colon\X\to[0,\infty)$
	and constants $\theta>\vartheta\ge1$ and $c>0$, such that
\begin{equation*}\Lyap(x)\,\ge\, c\,\bigl(L(x)\bigr)^{\theta}\,,\quad\text{and}\quad
	\phi\circ\Lyap(x)\,\ge\, c\,\bigl(L(x)\bigr)^\vartheta\qquad \forall\,x\in\X\,.\end{equation*}
	In addition, suppose that $\{X(t)\}_{t\in\T}$ admits an invariant
	$\uppi\in\cP(\X)$ such that
$\int_{\X}\bigl(L(x)\bigr)^{\vartheta+\varepsilon}\,\uppi(\D x)=\infty$
	for some $\varepsilon\in(0,\theta-\vartheta)$.
	Then, for each $p\in[1,\vartheta]$, $\iota\in(0,\theta-\vartheta-\varepsilon)$
	and $x\in\X$, there exist a constant $\Bar c>0$ and a diverging increasing sequence
	$\{t_n\}_{n\in\NN}\subseteq\T$, depending on these parameters, such that
	\begin{equation}\label{ET1.2A}
	\sW_p(\updelta_x P_{t_n},\uppi)\,\ge\,\Bar c\,
	\bigl(t_n+\Lyap(x)\bigr)^{-\frac{\vartheta-p+\varepsilon+\iota}
		{(\theta-\vartheta-\varepsilon-\iota)p}}\qquad \forall\, n\in\NN\,.
	\end{equation}
\end{thm}

Note that the parameters $\theta$, $\vartheta$, $\varepsilon$, $p$ and $\iota$ are such that the exponent in the above expression is always strictly negative.
Obtaining lower bound for the convergence in the total variation norm is
discussed in \cite[Theorem~5.1 and Corollary~5.2]{Hairer-Lecture-notes-2016}.
Applications of \cref{T1.2} are discussed in \cref{S3}.


\subsection{Ergodicity of a class of L\'evy-type processes}

Here, we discuss ergodic properties of a class of Markov processes on
the Euclidean space $\RR^n$ (endowed with the standard Euclidean metric)
generated by a (L\'evy-type) operator
$\mathcal{L}\colon\mathcal{D}(\mathcal{L})\subseteq\cB(\RR^n)\to\cB(\RR^n)$ 
given by
\begin{equation}\label{OP-L}
\mathcal{L} f(x) \,=\,
\bigl\langle b(x),\nabla f(x)\bigr\rangle
+\frac{1}{2}\trace\bigl(a(x)\nabla^2f(x)\bigr)
+\int_{\RR^n}\mathfrak{d}_1 f(x;y)\upnu(x,\D{y})\,,\qquad x\in\RR^n\,.
\end{equation}
Here, $b=(b_i)_{i=1,\dots,n}\colon\RR^n\to\RR^n$ is Borel measurable,
$a=(a_{ij})_{1\le i,j\le n}\colon\RR^n\to\RR^{n\times n}$ is a symmetric
non-negative definite $n \times n$ matrix-valued Borel measurable function,
$\upnu(x, \D y)$ is a nonnegative Borel kernel on
$\RR^n\times\mathfrak{B}(\RR^n)$, called the L\'evy kernel,
satisfying
\begin{equation*}\upnu(x,\{0\})\,=\,0\,,\quad\text{and}\quad \int_{\RR^n}\bigl(1\wedge|y|^2\bigr)
\,\upnu(x,\D y)\,<\,\infty\qquad \forall\,x\in\RR^n\,,\end{equation*}
and
\begin{equation*}
\mathfrak{d}_1 f(x;y)\,\df\,f(x+y)-f(x)-\Ind_{\sB}(y)
\langle y,\nabla f(x)\rangle\,,\qquad x,y\in\RR^n\,,\quad f\in C^1(\RR^n)\,.
\end{equation*}
The symbol $\mathcal{D}(\mathcal{L})$ stands for the domain of $\mathcal{L}$, i.e.
the set of functions $f\in\cB(\RR^n)$ for which \cref{OP-L} is well defined,
$\langle\cdot,\cdot \rangle$ and $\abs{\cdot}$ denote the standard inner product
and the corresponding Euclidean norm on $\RR^n$, $\trace M$
stands for the trace
of a square matrix $M$, and $\nabla^2f(x)$ denotes the Hessian of $f\in C^2(\RR^n)$.
An open (resp.\ closed) ball of radius $r>0$ centered at $x$ is denoted by $\sB_r(x)$
(resp.\ $\overline\sB_r(x)$).
If $x=0$, we write $\sB_r$ (resp.\ $\overline\sB_r$), and the unit open (resp.\ closed)
ball centered at $0$ is denoted by $\sB$ (resp.\ $\overline\sB$). 
Observe that $C_b^2(\RR^n)\subseteq\mathcal{D}(\mathcal{L})$, where $C_b^k(\RR^n)$,
$k\ge0$, denotes the space of $k$ times differentiable functions such that
all derivatives up to order $k$ are bounded. We also denote by
$\norm{M}\df \bigl(\trace MM'\bigr)^{\nicefrac{1}{2}}$
the Hilbert--Schmidt norm of a matrix $M$,
where $M'$ stands for the transpose of $M$.

We introduce the following assumption:

\begin{itemize}
	\item[\hypertarget{MP}{\textbf{(MP)}}] 
	There exists a conservative strong Markov process $\process{X}$ with
	c\`{a}dl\`{a}g sample paths such that 
	\begin{equation}\label{MART}
	M_f(t)\,\df\,f\bigl(X(t)\bigr)-f\bigl(X(0)\bigr)
	-\int_0^t\mathcal{L}f\bigl(X(s)\bigr)\,\D s\,,
	\qquad t\ge0\,,
	\end{equation}
	is a $\Prob_{x}$-martingale (with respect to $\{\mathcal{F}_t\}_{t\ge0}$)
	for any $f\in C_c^\infty(\RR^n)$
	(the space of smooth functions with compact support). 
\end{itemize}

Define
\begin{equation*}
	q(x,\xi)\,\df\,-i\langle\xi,b(x)\rangle+\frac{1}{2}\langle\xi,a(x)\xi\rangle
	+\int_{\RR^n}\bigl(1-\E^{i\langle\xi,y\rangle}+i\langle\xi,y\rangle
	\Ind_{\sB}(y)\bigr)\upnu(x,\D y)\,,\qquad x,\xi\in\RR^n\,,
\end{equation*}
and observe that
\begin{equation*}\mathcal{L}f(x)\,=\,-\int_{\RR^n}\E^{i\langle\xi,x\rangle}q(x,\xi)\Hat f(\xi)\,
\D \xi\end{equation*} for all $x\in\RR^n$ and $f\in C_c^\infty(\RR^n),$
where $\Hat f(\xi)\df(2\pi)^{-n}\int_{\RR^n}\E^{-i\langle\xi,x\rangle}f(x)\,\D x$
denotes the Fourier transform of $f(x)$.
In other words, $\mathcal{L}$ is a pseudo-differential operator with symbol
$q(x,\xi)$. According to \cite[Theorem~1.1]{Franziska1},
\hyperlink{MP}{\ttup{MP}} is satisfied if 
\begin{itemize}
	\item[\hypertarget{LB}{\textbf{(LB)}}]
	The functions $b(x)$, $a(x)$, and
$x\mapsto\int_{\RR^n}\bigl(1\wedge|y|^2\bigr)\,\upnu(x,\D y)$
	are locally bounded.
\smallskip
	\item[\hypertarget{SG}{\textbf{(SG)}}]
	$x\mapsto q(x,\xi)$ is continuous for all $\xi\in\RR^n$,
	and $q(x,\xi)$ is locally uniformly continuous at $\xi=0$, i.e.
	\begin{equation*}
		\lim_{\rho\to\infty}\,\sup_{x\in\sB_\rho}\,\sup_{\xi\in\sB_{\nicefrac{1}{\rho}}}\,
		\babs{q(x,\xi)}\,=\,0\,.
	\end{equation*}
\end{itemize}

Observe that the second condition in \hyperlink{SG}{\ttup{SG}}
essentially means that the coefficients $b(x)$, $a(x)$, and $\upnu(x,\D y)$
have a sublinear growth. Namely, it is satisfied if
\begin{align*}
	\lim_{\rho\to\infty}\,\Biggl(\frac{\sup_{x \in \sB_\rho}\abs{b(x)}}{\rho}
	+\frac{\sup_{x \in \sB_\rho}\abs{a(x)}}{\rho^2}
	&+\frac{\sup_{x \in \sB_\rho}\int_{\sB}\abs{y}^2\,\upnu(x,\D y)}{\rho^2}\\&
	+\sup_{x\in\sB_\rho}\,\sup_{\xi\in\sB_{\nicefrac{1}{\rho}}}\, \int_{\sB^c}
	\bigl(1-\E^{i\langle\xi,y\rangle}\bigr)\upnu(x,\D y)\Biggr)\,=\,0\,.
\end{align*}
In order to allow linear growth of the coefficients, we replace
\hyperlink{LB}{\ttup{LB}} and \hyperlink{SG}{\ttup{SG}} by
\begin{itemize}
	\item[\hypertarget{LG}{\textbf{(LG)}}]
	$\mathcal{L}\bigl(C_c^\infty(\RR^n)\bigr)\subseteq C_\infty(\RR^n)$, $x\mapsto q(x,\xi)$
	is continuous for all $\xi\in\RR^n$, and 
	\begin{equation*}
		\limsup_{\abs{x}\to\infty}\,\sup_{\xi\in\sB_{\nicefrac{1}{\abs{x}}}}\,\babs{q(x,\xi)}
		\,<\,\infty
	\end{equation*}
(see \cite[Corollary~3.2]{Franziska3}).
\end{itemize}

Here, $C_\infty(\RR^n)$ stands for the space of continuous functions vanishing
at infinity. Clearly, the last condition in \hyperlink{LG}{\ttup{LG}} follows from 
\begin{equation}\label{LGC}
\limsup_{|x|\to\infty}\,\left(\frac{|b(x)|}{|x|}+\frac{\lVert a(x)\rVert}{|x|^2}
+\frac{\int_{\sB}|y|^2\upnu(x,\D y)}{|x|^2}+\upnu(x,\sB^c)\right)\,<\,\infty\,.
\end{equation}
Let us also remark that due to \cite[Theorem~A1]{Franziska2} the map
$x\mapsto q(x,\xi)$ is continuous for all $\xi\in\RR^n$
if $b(x)$ and $a(x)$ are continuous, and for any $r>0$, $x\in\RR^n$ and
$f\in C_c(\RR^n\setminus\{0\})$,
\begin{align*}&
	\lim_{\rho\to\infty}\,\sup_{y\in\sB_r}\, \upnu(y,\sB_\rho^c)\,=\,0\,,
	\qquad \lim_{\rho\to0}\,\sup_{y\in \sB_r}\, \int_{\sB_\rho}|z|^2\upnu(y,\D z)\,=\,0\,,
	\intertext{and}
	&\lim_{y\to x}\, \int_{\RR^n}f(z)\,\upnu(y,\D z)\,=\,\int_{\RR^n}f(z)\,\upnu(x,\D z)\,.
\end{align*}
Furthermore, under the continuity of $x\mapsto q(x,\xi)$ (for all $\xi\in\RR^n$)
in the same reference it has been shown that
$\mathcal{L}\bigl(C_c^\infty(\RR^n)\bigr)\subseteq C_b(\RR^n)$.
In addition, if
\begin{equation*}\lim_{\abs{x}\to\infty}\, \upnu\bigl(x,\sB_r(-x)\bigr)\,=\,0\qquad \forall\,r>0\,,\end{equation*}
we easily see that $\mathcal{L}\bigl(C_c^\infty(\RR^n)\bigr)\subseteq C_\infty(\RR^n)$.

\begin{definition}\label{D1.1}
	Let $\mathcal{M}_+$ denote the class of positive definite matrices in
	$\RR^{n\times n}$.
	For $Q\in\mathcal{M}_+$, let $\abs{x}_Q\df\langle x,Qx\rangle^{\nicefrac{1}{2}}$ for
	$x\in\RR^n$,
	and $\chi_Q\in C^\infty(\RR^n)$ be some nonnegative, symmetric
	convex function such that
	$\chi_Q(x)=\abs{x}_Q$ for $x\in\cB^c$.
	For $Q\in\mathcal{M}_+$ and $\zeta>0$, we define
	\begin{equation*}
		\Lyap_{Q,\zeta}(x) \,\df\, \bigl(\chi_Q(x)\bigr)^\zeta\,,
		\quad\text{and}\quad
		\widetilde\Lyap_{Q,\zeta}(x)\,\df\,\E^{\zeta\chi_Q(x)}\,,\qquad x\in\RR^n\,.
	\end{equation*}
	Further, let
	\begin{equation*}
		\Theta_\upnu\,\df\,\left\{\theta\ge0\,\colon\sup_{x\in\RR^n}\int_{\RR^n}
		\bigl(\abs{y}^{2}\,\Ind_{\sB}(y)+\abs{y}^{\theta}\,\Ind_{\sB^c}(y)\bigr)\,\upnu(x,\D{y})
		<\infty\right\}\,,
	\end{equation*}
	and when $\Theta_\upnu\neq\emptyset$, let $\theta_\upnu\df\sup \Theta_\upnu$.
\end{definition}

We now discuss ergodic properties of the L\'evy-type process $\process{X}$.

\begin{thm}\label{T1.3}
	Assume
	\hyperlink{LB}{\ttup{LB}} and \hyperlink{MP}{\ttup{MP}}, and suppose that
	$\process{X}$ is irreducible and aperiodic, and that every compact set is petite
	for $\process{X}$. 
	Then the following hold.
	\begin{enumerate}
		\item[\ttup{i}]
		If $\theta_\upnu>0$,
		\begin{equation}\label{ET1.3A}
		\lim_{r\to\infty}\,\sup_{x\in\RR^n}\,
		\int_{\sB^c_r}\abs{y}^\theta\,\upnu(x,\D y)\,=\,0
		\end{equation}
		for some $\theta\in(0,\theta_\upnu]\cap\Theta_\upnu$, and there exist
		$Q\in\mathcal{M}_+$ and $\vartheta\in[0\vee(2-\theta),2)$ such that
\begin{equation*}
\limsup_{\abs{x}\to\infty}\frac{\lVert a(x)\rVert}{|x|^\vartheta}\,=\,0\,,
\quad\text{and}\quad
			\limsup_{\abs{x}\to\infty}\,\frac{\bigl\langle b(x)
				+\Ind_{[1,\infty)}(\theta)\int_{\sB^c}\\
				y\,\upnu(x,\D{y}),
				Qx\bigr\rangle}{\abs{x}^\vartheta} \,<\, 0\,,
		\end{equation*}
		then $\process{X}$ admits a unique invariant
		$\uppi\in\mathcal{P}_{\theta-2+\vartheta}(\RR^n)$.
		In addition, if $\theta-3+\vartheta\ge0$, then \cref{T1.1}\,\ttup{i} and \ttup{ii}
		hold with $\Lyap(x)=\Lyap_{Q,\theta}(x)+1$,
		$\phi(t)=t^{\nicefrac{(\theta-2+\vartheta)}{\theta}}$ and $\eta=\theta-2+\vartheta$.
		
\item[\ttup{ii}]
		If $\theta_\upnu>0$, \cref{ET1.3A} holds for some
		$\theta\in(0,\theta_\upnu]\cap\Theta_\upnu$, and there exists $Q\in\mathcal{M}_+$
		such that
		\begin{equation*}
			\limsup_{\abs{x}\to\infty}\,\frac{\lVert a(x)\rVert}{|x|^2}\,=\,0\,,
\quad\text{and}\quad
\limsup_{\abs{x}\to\infty}\,\frac{\bigl\langle b(x)+\Ind_{[1,\infty)}(\theta)\,
\int_{\sB^c}y\,\upnu(x,\D{y}),
				Qx\bigr\rangle}{\abs{x}^2} \,<\, 0\,,
		\end{equation*}
then $\process{X}$ admits a unique invariant $\uppi\in\mathcal{P}_{\theta}(\RR^n)$.
		In addition, if $\theta\ge1$,
		then the conclusion of \cref{T1.1}\,\ttup{iii} holds with 
		$\Lyap(x)=\Lyap_{Q,\theta}(x)+1$ and $\eta=\theta$.
		
		\item[\ttup{iii}]
		Suppose that $a(x)$ is bounded, and there exist $\theta>0$
		and $Q\in\cM_+$, such that
		\begin{equation}\label{ET1.3B}
		\sup_{x\in\RR^n}\,\int_{\RR^n}
		\bigl(\abs{y}^2\Ind_{\sB}(y)+\E^{\theta\abs{y}}\Ind_{\sB^c}(y)\bigr)\,
		\upnu(x,\D y)\,<\,\infty\,,
		\end{equation}
		and
		\begin{equation*}
			\limsup_{\abs{x}\to\infty}\,\frac{\bigl\langle b(x)+\int_{\sB^c}y\,\upnu(x,\D{y}),
				Qx\bigr\rangle}{\abs{x}} \,<\, 0\,.
		\end{equation*}
		Then the conclusion of \cref{T1.1}\,\ttup{iii} holds with 
		$\Lyap(x)=\widetilde\Lyap_{Q,\zeta}(x)$ for any $\zeta>0$
		sufficiently small
		and any $\eta\ge1$.
	\end{enumerate}
\end{thm}

Irreducibility and aperiodicity are crucial structural properties of the
underlying process in \cref{T1.1,T1.2,T1.3}. Roughly speaking,
they ensure that the process does not show singular behavior in its motion,
and together with the Foster-Lyapunov condition in \cref{ET1.1A}
(which ensures controllability of the $\phi\circ\Phi^{-1}$-modulated moment of return-times
to the petite set $C$, see \cite[Theorem~4.1]{Douc-Fort-Guilin-2009}) 
they lead to the ergodic properties stated.

Under an asymptotic flatness (uniform dissipativity) property
(see \cref{ET1.4A}),
we use a completely different approach to this problem, the
so-called synchronous coupling method (see \cite[Example 2.16]{Chen-Book-2005}
for details), to obtain ergodic properties for a class of It\^{o}
processes which are not necessarily irreducible and aperiodic.
Recall that an It\^{o} process is a 
solution to a stochastic differential equation (SDE) of the following form 
\begin{equation}\label{SDE}
\begin{aligned}
X(t)\,=\,x&+\int_0^tb\bigl(X(s)\bigr)\,\D s+\int_0^t\sigma\bigl(X(s)\bigr)\,\D B_s\\
&+\int_0^t\int_{\{w\colon |k(X(s-),w)|<1\}} k(X(s-),v)\,
\bigl(\nu_p(\D v,\D s)-\nu(\D v)\, \D s\bigr)\\
&+\int_0^t\int_{\{w\colon|k(X(s-),w)|\ge1\}} k(X(s-),v)\, \nu_p(\D v,\D s)\,,
\qquad (t,x)\in[0,\infty)\times\RR^n\,, 
\end{aligned}
\end{equation}
where $b\colon\RR^n\to\RR^n$, $\sigma\colon\RR^n\to\RR^{n\times n}$ and
$k\colon\RR^n\times\RR\to\RR^n$ are Borel measurable,
$\process{B}$ is a standard $n$-dimensional Brownian motion,
and $\nu_p(\D v,\D s)$
is a Poisson random measure on
$\mathfrak{B}(\RR)\otimes\mathfrak{B}\bigl([0,\infty)\bigr)$,
with intensity measure $\nu(\D v)\,\D s$ (a $\sigma$-finite
measure on $\mathfrak{B}(\RR)\otimes\mathfrak{B}(\RR)$).
According to \cite[Theorem~3.33]{Cinlar-Jacod-1981}, every It\^{o} process is a 
semimartingale Hunt process.
In particular, it is a conservative strong Markov process with c\`{a}dl\`{a}g
sample paths. Conversely, again by \cite[Theorem~3.33]{Cinlar-Jacod-1981},
for every $n$-dimensional semimartingale Hunt process $\process{X}$,
and every $\sigma$-finite nonfinite and nonatomic measure $\nu(\D v)$ on
$\mathfrak{B}(\RR)$, there exist $b(x)$, $\sigma(x)$, $k(x,v)$, $\process{B}$,
and $\nu_p(\D v, \D s)$ as above (possibly defined on an enlargement of the
initial stochastic basis), such that $\process{X}$ satisfies \cref{SDE}.
By setting
\begin{equation*}\upnu_p(\D y,\D s)\,=\,\nu_p\bigl(\{(v,u)\in\RR\times[0,\infty)
\colon(k(X(u-),v),u)\in(\D y,\D s)\}\bigr)\,,\end{equation*}
and 
\begin{equation*}\upnu(x,\D y)\,=\,\nu\bigl(\{u\in\RR\colon k(x,u)\in\D y\}\bigr)\,,\end{equation*}
\cref{SDE} reads as
\begin{equation*}
	\begin{aligned}
X(t)\,=\,x+\int_0^tb\bigl(X(s)\bigr)\,\D s+\int_0^t\sigma\bigl(X(s)\bigr)\,\D B_s
		&+\int_0^t\int_{\sB} y\,\bigl(\upnu_p(\D y,\D s)-\upnu(X(s-),\D y)\,\D s\bigr)\\
		&+\int_0^t\int_{\sB^c} y\, \upnu_p(\D y,\D s)\,,
		\qquad (t,x)\in[0,\infty)\times\RR^n\,.
	\end{aligned}
\end{equation*}
Set $a(x)\df\sigma(x)\sigma(x)'$, and let $\mathcal{L}$ be as in \cref{OP-L}.
According to \cite[Theorem~II.2.42]{Jacod-Shiryaev-2003}
(with $h(x)=x\Ind_{\sB}(x)$), for any $f\in C_b^2(\RR^n)$, the process $\process{M_f}$,
defined as in \cref{MART}, is a $\Prob_x$-local martingale for every $x\in\RR^n$. 
In addition, if \hyperlink{LB}{\ttup{LB}} holds true, then $\process{M_f}$ is a
$\Prob_x$-local martingale for every $f\in C_c^\infty(\RR^n)$ and every $x\in\RR^n$,
i.e. \hyperlink{MP}{\ttup{MP}} is satisfied.

For $x,z\in\RR^{n}$ define
\begin{equation*}
	\varDelta_{z}b(x)\,\df\, b(x+z)-b(x)\,,\quad
	\varDelta_{z}\sigma(x)\,\df\,\sigma(x+z)-\sigma(x)\,,\quad
\varDelta_{z}\upnu(x,\D y)\,\df\, \upnu(x+z,\D y)-\upnu(x,\D y)\,,
\end{equation*}
\begin{equation*}\varDelta_{z}\Tilde b(x)\,\df\, \varDelta_{z}b(x)
	+\int_{\sB^c}y\, \varDelta_{z}\upnu(x,\D y)\,,\quad 
	\text{and} \quad\Tilde{a}(x;z)\df \varDelta_{z}\sigma(x) 
	\varDelta_{z}\sigma(x)'\,.\end{equation*}
If $b(x)\equiv b$ (resp.\ $\sigma(x)\equiv\sigma$, or
$\upnu(x,\D y)\equiv\upnu(\D y)$), then of course $\varDelta_{z}b(x)$
(resp.\ $\varDelta_{z}\sigma(x)$, or $\varDelta_{z}\upnu(x,\D y)$) is equal to zero.

\begin{thm}\label{T1.4}
	Assume that $b(x)$ and $a(x)$ are locally bounded and satisfy the linear growth
	condition in \cref{LGC}, and that $\upnu(x,\D y)$ is such
	that $2\in\Theta_\upnu$. 
	If for some $p\in[2,\theta_\upnu]\cap\Theta_\upnu$ there exist $Q\in\cM_+$,
	and a $\sigma$-finite nonfinite and nonatomic measure $\nu(\D v)$ on
	$\mathfrak{B}(\RR)$ such that \cref{SDE} admits a unique strong solution $\process{X}$,
	and
	\begin{equation}
	\begin{aligned}\label{ET1.4A} 
	2\,\bigl\langle \varDelta_{z}\Tilde b(x),Qz\bigr\rangle& +\trace\,
	\bigl(\Tilde{a}(x;z) Q\bigr)+(p-2)\,\bnorm{\sqrt{Q}\,\varDelta_{z}\sigma(x)}^{2}\\
	& +2^{p-3}\bigl(1+(p-2)\norm{Q^{-1}}\bigr)
	\int_{\RR}\babs{k(x+z,v)-k(x,v)}_Q^2\, \nu(\D v)\\
	& +\frac{2^{p-2}}{p(p-1)}\bigl(1+(p-2)\norm{Q^{-1}}\bigr)\abs{z}_Q^{2-p}
	\int_{\RR}\babs{k(x+z,v)-k(x,v)}_Q^{p}\, \nu(\D v)
	 \,\le\, -\frac{2\,c(p)}{p}\abs{z}_Q^2
	\end{aligned}
	\end{equation}
	for some $c(p)>0$ and all $x,z\in\RR^n$, where $k\colon\RR^n\times\RR\to\RR^n$
	is given in \cref{SDE},
	then
	\begin{equation}\label{ET1.4B}
	\sW_p(\updelta_xP_t,\updelta_yP_t)\,\le\,
	\left(\frac{\overline\lambda_Q}{\underline\lambda_Q}\right)^{\nicefrac{1}{2}}|x-y|\,
	\E^{-\frac{c(p)t}{p}}
	\end{equation} for all $t\ge0$ and $x,y\in\RR^n$,
	where $\overline\lambda_Q$ ($\underline\lambda_Q$) stands for the largest (smallest)
	eigenvalue of $Q$.
	Furthermore, $\process{X}$ admits a unique invariant $\uppi\in\cP_{p}(\RR^n)$, and 
	\begin{equation}\label{ET1.4C}
	\sW_p(\upmu P_t,\uppi)\,\le\,
	\left(\frac{\overline\lambda_Q}{\underline\lambda_Q}\right)^{\nicefrac{1}{2}}
	\sW_p(\upmu,\uppi)\,\E^{-\frac{c(p)t}{p}}
	\end{equation}
	for all $t\ge0$ and $\upmu\in\cP_p(\RR^n)$.
	
	In addition, if $\sigma(x)\equiv\sigma$, a constant,
	$\upnu(x,\D y)\equiv\upnu(\D y)$,
	$1\in\Theta_\upnu$, and \cref{ET1.4A} holds for some
	$p\in[1,\theta_\upnu]\cap\Theta_\upnu$, then \cref{ET1.4B,ET1.4C} remain valid. 
\end{thm}

We remark that  ergodic properties of a Markov process with respect to
the $\sW_p$-distance are invariant under the Bochner's random time-change method. 
Recall that a subordinator $\process{S}$ is a nondecreasing L\'{e}vy process
on $\left[0,\infty\right)$ with Laplace transform
$\mathbb{E}\bigl[\E^{-uS_t}\bigr] = \E^{-t\psi(u)}$,
$u,t\geq 0$.
The characteristic (Laplace) exponent $\psi\colon(0,\infty)\to(0,\infty)$
is a Bernstein function, i.e. it is of class $C^\infty$ and
$(-1)^n\psi^{(n)}(u)\ge0$ for all $n\in\NN$.
It is well known that every Bernstein function admits a unique
(L\'{e}vy-Khintchine) representation 
\begin{equation*}\psi(u)\,=\,b_Su+\int_{(0,\infty)}(1-\E^{-uy})\,\upnu_S(\D y) \qquad \forall\,u\ge0\,,\end{equation*}
where $b_S\geq0$ is the drift parameter and $\upnu_S(\D y)$ is a L\'{e}vy measure,
i.e. a Borel measure on $\mathfrak{B}\bigl((0,\infty)\bigr)$ satisfying
$\int_{(0,\infty)}(1\wedge y)\,\upnu(\D y)<\infty$.
For additional reading on subordinators and Bernstein functions we refer the reader to the
monograph \cite{Schilling-Song-Vondracek-Book-2012}.
Suppose $\process{X}$ is a Markov process on $\bigl(\X,\mathfrak{B}(\X)\bigr)$ with
transition kernel $p(t,x,\D y)$, and let
$\process{S}$ be a subordinator with characteristic exponent $\psi(u)$,
independent of $\process{X}$. 
The process $X^{\psi}(t)\df X\bigl(S(t)\bigr)$, $t\ge0$, obtained from $\process{X}$ by 
a random time change through $\process{S}$, is referred to as the subordinate
process $\process{X}$ with subordinator $\process{S}$ in the sense of Bochner.
It is easy to see that $\process{X^\psi}$ is again a Markov process with
transition kernel
\begin{equation*}p^\psi(t,x,\D y)\,=\,\int_{\left[0,\infty\right)} p(s,x,\D y)\,\upmu_t(\D s)\,,
\qquad t\ge0\,,\quad x\in\RR^n\,,\end{equation*}
where $\upmu_t(\cdot)=\mathbb{P}(S(t)\in\cdot)$.
It is also elementary to check that if $\uppi(\D x)$ is an invariant measure for
$\process{X}$, then it is also invariant for the subordinate process $\process{X^\psi}$. 

\begin{proposition}\label{P1.1} 
	Assume that $\process{X}$ admits an invariant $\uppi\in\mathcal{P}(\X)$ such that
	$\sW_{p}(\updelta_x P_t,\uppi)\le c(x)\,r(t)$ for some $p\ge1$, and all $t\ge0$
	and $x\in\X$, where $r\colon[0,\infty)\to[1,\infty)$ is Borel measurable,
	and $c\colon\X\to[0,\infty)$. 
	Then,
	\begin{equation*}
		\sW_{p}(\updelta_x P^\psi_t,\uppi) \,\le\, c(x)\,r_\psi(t)\qquad\forall\,
		(t,x)\in[0,\infty)\times\X\,,
	\end{equation*}
	where
	$r_\psi(t)\df\Bigl(\mathbb{E}\Bigl[\bigl(r(S(t))\bigr)^p\Bigr]\Bigr)^{\nicefrac{1}{p}}$. 
\end{proposition}

Ergodic properties of Markov processes under subordination in the $f$-norm
are discussed in
\cite{Deng-2020,Deng-Schilling-Song-2017,Deng-Schilling-Song-2017-Errata}.


\subsection{Literature review}

Our work contributes to the understanding of the ergodic properties of
Markov processes.
Most of the existing literature focuses on characterizing the exponential
or subexponential ergodicity under the $f$-norm, and in particular the total
variation norm, see 
\cite{AFV15,CF09,Douc-Fort-Guilin-2009,DFMS-04,Down-Meyn-Tweedie-1995,Fort-03,
	Fort-Roberts-2005,G-09,LZZ-10,Meyn-Tweedie-Book-2009,Meyn-Tweedie-AdvAP-II-1993,
	Meyn-Tweedie-AdvAP-III-1993,TT-94}
and the references therein. 
However, there have been some recent developments in understanding ergodic properties
of Markov processes (both continuous and discrete time) under the Wasserstein distances;
see \cite{Butkovsky-14,DFM-16,friesen2019exponential,Kulik-Book-2018,
	liang2019exponential,Lazic-Sadric-2020,Eberle-2011,Eberle-2015,Wang-16,
	Luo-16,majka2017coupling}. 
As already mentioned, exponential and subexponential convergence rates in the
$\sW_1$-distance for general Markov processes that are (possibly) not irreducible
or aperiodic are established in \cite{Butkovsky-14,DFM-16,Kulik-Book-2018},
under the Foster-Lyapunov condition in \cref{ET1.1A}, contractivity of the
underlying metric, and smallness of sublevel sets of the corresponding Lyapunov function. 
Using the coupling approach, the authors in \cite{Eberle-2011,Eberle-2015,Luo-16}
studied exponential ergodicity with respect to a class of Wasserstein distances
for SDEs driven by an additive Brownian noise term and
a drift term satisfying an asymptotic flatness property at infinity.
Under the same assumption on the drift term,
these results have been extended in \cite{majka2017coupling,Wang-16}
to allow for more general additive L\'evy noises.
Subexponential ergodicity with respect to the $\sW_p$-distance for stochastic
differential equations driven by an additive L\'evy noise term, with a drift term
satisfying asymptotic flatness property at zero, has been studied
in \cite{Lazic-Sadric-2020}.
By combining the Foster-Lyapunov method with the coupling approach, exponential
ergodicity with respect to a class of $f$-norms and Wasserstein distances
(given in terms of the underlying Lyapunov function)
is established in \cite{liang2019exponential}
for  a class of Mckean-Vlasov SDE with L\'evy noise. 
Lastly,  exponential ergodicity with respect
to the $\sW_1$-distance for one-dimensional positive-valued stochastic
differential equations  with jumps and the drift term satisfying 
asymptotic flatness property has been studied in \cite{friesen2019exponential}.

Our results on both exponential and subexponential ergodicity under the $\sW_p$-distance
contribute to this active research topic. Of particular interest is the 
result obtained in \cref{T1.2} which seems to be completely new in the literature,
and which, in some cases, allows one to conclude that the obtained upper bound on
the rate of convergence is sharp.   

As we have already remarked, irreducibility and aperiodicity are crucial
structural properties of the underlying process used in \cref{T1.1,T1.2,T1.3}.
There is a vast literature on these, and related questions such as the strong
Feller property and heat kernel estimates of Markov processes. 
In particular, we refer the readers to 
\cite{ari,chen-chen-wang,chen-hu-xie-zhang,chen-zhang,chen-zhang2,grz-sz,
	kim-lee,kim-sig-vond,Knopova-Schilling-2012,Knopova-Schilling-2013,
	Kolokoltsov-2000,Kolokoltsov-Book-2011,Kwon-Lee-1999, Pang-Sandric-2016,
	Sandric-TAMS-2016,Stroock-1975}
for the case of a class of Markov L\'evy-type processes with bounded coefficients,
and to
\cite{Arapostathis-Pang-Sandric-2019,Bass-Cranston-1986,Ishikawa-2001,
	Knopova-Kulik-2014,Lazic-Sadric-2020,Masuda-2007,Masuda-Erratum-2009, 
	Picard-1996,Picard-Erratum-2010,Sadric-Valentic-Wang-2019,stramer,Xi-Zhu-2019}
for the case of a class of It\^{o} processes. 

Recall that the Foster-Lyapunov condition in \cref{ET1.1A} implies that 
for any $\varepsilon>0$ the $\phi\circ\Phi^{-1}$-modulated moment of the
$\varepsilon$-shifted hitting time
$\tau_C^\varepsilon\df\inf\{t\ge\delta\colon X(t)\in C\}$ of $\process{X}$
of $C$ (with respect to $\mathbb{P}_x$) is finite and controlled by $\Lyap(x)$
(see \cite[Theorem~4.1]{Douc-Fort-Guilin-2009}). 
However, this property in general does not immediately imply ergodicity of
$\process{X}$.
Namely, we also need to ensure that a similar property holds for any
other ``reasonable'' set. 
If $\process{X}$ is irreducible with irreducibility measure $\upvarphi(\D x)$,
then indeed for any $\varepsilon>0$ the $\phi\circ\Phi^{-1}$-modulated moment
of $\tau_B^\varepsilon$, for any $B\in\mathfrak{B}(\X)$
with $\upvarphi(B)>0$, is again finite and controlled by $\Lyap(x)$
(see \cite[the discussion after Theorem~4.1]{Douc-Fort-Guilin-2009}).
However, $\process{X}$ can also show certain cyclic behavior which destroys ergodicity
(see \cite[Section~5]{Meyn-Tweedie-AdvAP-II-1993} and
\cite[Chapter 5]{Meyn-Tweedie-Book-2009}).
By assuming aperiodicity, which excludes this type of behavior, (sub)exponential
ergodicity in the $\sW_p$-distance of $\process{X}$ follows as discussed in
\Cref{T1.1}, and in the $f$-norm as discussed in \cite[Theorem~1]{Fort-Roberts-2005}.

\subsection{Organization of the article}

In \cref{S2}, we give the proofs of \cref{T1.1,T1.2,T1.3,T1.4,P1.1} together
with some auxiliary lemmas.
Applications of the main results to several classes of Markov processes,
including Langevin tempered diffusion processes, Ornstein-Uhlenbeck processes with jumps, piecewise Ornstein-Uhlenbeck processes with jumps
under constant and stationary Markov controls, state-space models,
and backward recurrence time chains, are contained in \cref{S3}.


\section{Proofs of the main results}\label{S2}

We start with the proof of \cref{T1.1}.

\begin{proof}[Proof of \cref{T1.1}]
	We consider the case when $\T=\RR_+$ only.
	The case when $\T=\Z_+$ proceeds in an analogous way, by employing the results
	from \cite[Theorem~2.8]{DFMS-04} and \cite[Theorem~15.0.2]{Meyn-Tweedie-Book-2009}.
	
	First, under the assumptions of the theorem, it has been shown
	in \cite[Proposition~3.1]{Douc-Fort-Guilin-2009} and
	\cite[Theorem~4.2]{Meyn-Tweedie-AdvAP-III-1993} 
	that $\process{X}$ admits a unique invariant $\uppi\in\cP_{\phi\circ \Lyap}(\X)$.
	This, together with \cref{ET1.1B}, implies that $\uppi\in\cP_{\eta}(\X)$.
	We continue now with the proof of part \ttup{i}. 
	By the Kantorovich-Rubinstein theorem, we have
	\begin{equation*}
		\sW_1(\upmu_1,\upmu_2) \,=\, \sup_{\{f\colon\Lip(f)\le1\}}\,
		\biggl|\int_{\X} f(x)\bigl(\upmu_1(\D{x})-\upmu_2(\D{x})\bigr)\biggr|
		\qquad \forall\,\upmu_1,\upmu_2\in\mathcal{P}_1(\X)\,,
	\end{equation*}
	where the supremum is taken over all Lipschitz continuous functions
	$f\colon\X\to\RR$ with Lipschitz constant $\Lip(f)\le1$.
	We apply \cite[Theorem~3.2]{Douc-Fort-Guilin-2009},
\begin{equation*}
r_*(t)\,=\,\phi\circ\Phi^{-1}(t)\,,\quad
f_*(x)\,=\,\phi\circ\Lyap(x)\,,\quad
\Psi_1(z)\,=\,z^{\nicefrac{(\eta-1)}{\eta}}\,,\quad\text{and}\quad
\Psi_2(z)\,=\,c^{-\nicefrac{1}{\eta}}z^{\nicefrac{1}{\eta}}\,.
\end{equation*}
	Note that if $f\colon\X\to\RR$ is such that $\Lip(f)\le1$ and $f(x_0)=0$, then
	$\abs{f(x)}\le \mathsf{d}(x,x_0)\le \Psi_2\circ f_*(x)$. 
	Thus
	\begin{equation*}
		\sup_{f\colon\Lip(f)\le1}\,
		\biggl|\int_{\X} f(x)\bigl(\upmu_1(\D{x})-\upmu_2(\D{x})\bigr)\biggr|
		\,\le\,
		\sup_{\abs{f} \,\le\, \Psi_2\circ f_*\vee1}\,
		\biggl|\int_{\X} f(x)\bigl(\upmu_1(\D{x})-\upmu_2(\D{x})\bigr)\biggr|
		\,=\, \norm{\upmu_1-\upmu_2}_{\Psi_2\circ f_*\vee1}
	\end{equation*}
(recall the definition of the  $f$-norm in \cref{EFNORM}).
	Now, from
	\cite[(3.5) and (3.6)]{Douc-Fort-Guilin-2009} we have
	\begin{align*}
		\bigl(\Psi_1\circ r_*(t)\vee1\bigr)\sW_1(\updelta_x P_t,\uppi)
		&\,\le\,\bigl(\Psi_1\circ r_*(t)\vee1\bigr)
		\norm{\updelta_x P_t-\uppi\,}_{\Psi_2\circ f_*\vee1}\,\le\, \Bar c\, \Lyap(x)\,,\\[5pt]
		\intertext{and}
		\int_0^\infty\bigl(\Psi_1\circ r_*(s)\vee1\bigr)
		\sW_1(\updelta_x P_s,\updelta_y P_s)\,\D s
		&\,\le\, \int_0^\infty\bigl(\Psi_1\circ r_*(s)\vee1\bigr)
		\norm{\updelta_x P_s-\uppi\,}_{\Psi_2(f_*)\vee1}\,\D s\\
		&\,\le\,\Bar c\, \bigl(\Lyap(x)+\Lyap(y)\bigr)\,,
	\end{align*}
	for some $\Bar c>0$, and all $t\ge0$ and $x,y\in\RR^n$, which proves
	\cref{ET1.1C,ET1.1D}, respectively.
	
	We next prove part \ttup{ii}.
	Applying \cref{ET1.1C} and \cite[(3.5)]{Douc-Fort-Guilin-2009} with
	$\Psi_1(z)=1$, and $\Psi_2(z)=z$, we obtain
	$\Exp_{x} \left[\mathsf{d}(X(t),x_0)^{\eta}\right] \le
	\overline{m}_{\eta}+\Breve{c}\,\Lyap(x)$,
	for some $\Breve{c}>0$, and all $t\ge0$ and $x\in\X$.
	Hence
	\begin{equation}\label{EPT1.1A}
	\Exp_{x} \bigl[\mathsf{d}(X(t),x_0)^p\,\Ind_{\sB_{t}^c(x_0)}\bigl(X(t)\bigr)\bigr]
	\,\le\, t^{p-\eta}\,\bigl(\overline{m}_{\eta}+\Breve{c}\,\Lyap(x)\bigr)
	\qquad \forall\,(t,x)\in[0,\infty)\times\X\,.
	\end{equation}
	Further, for $t\ge0$, $z\in\X$,
	and $\Pi\in\mathcal{C}(\updelta_{z} P_t,\uppi)$, we have
	\begin{equation}\label{EPT1.1B}
\begin{aligned}
	 \int_{\X\times\X}\bigl(\mathsf{d}(x,y)\bigr)^p\,&\Pi(\D x,\D y)\\
	&\,=\,\int_{\sB_{t}(x_0)\times\sB_{t}(x_0)}\bigl(\mathsf{d}(x,y)\bigr)^p\,\Pi(\D x,\D y)
	+\int_{\bigl(\sB_{t}(x_0)\times\sB_{t}(x_0)\bigr)^c}
	\bigl(\mathsf{d}(x,y)\bigr)^p\,\Pi(\D x,\D y)\\
	&\,\le\,(2t)^{p-1}\int_{\X\times\X}\mathsf{d}(x,y)\,\Pi(\D x,\D y)
	+2^{p-1}\int_{\sB_{t}^c(x_0)}\bigl(\mathsf{d}(x,x_0)\bigr)^p
	\bigl[\updelta_{z}P_t(\D x)+\uppi(\D x)\bigr]\,.
	\end{aligned}
	\end{equation}
	Using \cref{EPT1.1A,EPT1.1B}, and the bound 
	$\int_{\sB_{t}^c(x_0)}\bigl(\mathsf{d}(x,x_0)\bigr)^p\,
	\uppi(\D x)\le t^{p-\eta}\,\overline{m}_{\eta}$, 
	we have
	\begin{equation*}
		\sW_p^p(\updelta_x P_t,\uppi)
		\,\le\,(2t)^{p-1}\,\sW_1(\updelta_x P_t,\uppi)
		+ 2^{p-1} t^{p-\eta}\,\bigl(2\overline{m}_{\eta}
		+\Breve{c}\,\Lyap(x)\bigr)\qquad \forall\,(t,x)\in[0,\infty)\times\X\,,
	\end{equation*}
	and combining this with \cref{ET1.1C} we obtain
	\begin{equation*}
		\left(1\vee \left(t^{\eta-p}\wedge t^{1-p}
		\bigl(r_*(t)\bigr)^{\nicefrac{(\eta-1)}{\eta}}\right)\right)\,
		\sW_p^p(\updelta_x P_t,\uppi)
		\,\le\, 2^{p-1} \Bar c\, \Lyap(x)
		+ 2^{p-1}\,\bigl(2\overline{m}_{\eta} +\Breve{c}\,\Lyap(x)\bigr)
	\end{equation*}
	for all $t\ge0$ and $x\in\X$,
	from which \cref{ET1.1E} follows with
	$\Tilde c = 2\max\{1,\Bar c,\Breve{c}\}^{\nicefrac{1}{p}}$.
	
	Moving on to the proof of part \ttup{iii}, note that 
	according to \cite[Proposition 6.1]{Meyn-Tweedie-AdvAP-II-1993},
	\cite[Theorem~4.2]{Meyn-Tweedie-AdvAP-III-1993}, and
	\cite[Theorem~5.2]{Down-Meyn-Tweedie-1995}, there
	exist constants $\mathring{c}>0$ and $\gamma>0$, such that
	\begin{equation}\label{EPT1.1C}
	\norm{\updelta_xP_t-\uppi\,}_{\Lyap} \,\le\,
	\mathring{c}\, \Lyap(x) \,\E^{-\gamma t}\qquad \forall\,(t,x)\in[0,\infty)\times\X\,.
	\end{equation}
	\Cref{ET1.1F} now follows from the Kantorovich-Rubinstein theorem and \cref{ET1.1B}.
	Let $p\in[1,\eta]$.
	First, from \cref{EPT1.1C} we obtain
	$\Exp_{x} \left[\mathsf{d}(X_t,x_0)^{\eta}\right] \le
	\overline{m}_{\eta}+\dot{c}\,\Lyap(x)$,
	for some $\dot{c}>0$, and all $t\ge0$ and $x\in\X$,
	which again implies \cref{EPT1.1A}.
	By \cref{EPT1.1A,EPT1.1B}, we have 
	\begin{equation*}
		\sW_p^p(\updelta_x P_t,\uppi)
		\,\le\, (2t)^{p-1}\,\sW_1(\updelta_x P_t,\uppi)
		+ 2^{p-1} t^{p-\eta}\,\bigl(2\overline{m}_{\eta}
		+\dot{c}\,\Lyap(x)\bigr)\qquad \forall\,(t,x)\in[0,\infty)\times\X\,,
	\end{equation*}
	and combining this with \cref{ET1.1F} we obtain
	\begin{equation*}
		(1\vee t^{\eta-p})\, \sW_p^p(\updelta_x P_t,\uppi)
		\,\le\, 2^{p-1} \check c\,\Lyap(x) + 2^{p-1} \,\bigl(2\overline{m}_{\eta}
		+\dot{c}\,\Lyap(x)\bigr)\qquad \forall\,(t,x)\in[0,\infty)\times\X\,,
	\end{equation*}
	from which \cref{ET1.1G} follows again with
	$\breve{c}= 2\max\{1,\check c,\dot{c}\}^{\nicefrac{1}{p}}$.
	This completes the proof.
\end{proof}

We proceed with the proof of \cref{T1.2}.

\begin{proof}[Proof of \cref{T1.2}]
	We again consider the case when $\T=\RR_+$ only.
	The case when $\T=\Z_+$ proceeds in a similar manner.
	
	Fix some $x_0\in\X$, $p\in[1,\vartheta]$ and
	$\iota\in(0,\theta-\vartheta-\varepsilon)$.
	For $s>0$, define $f_s\colon\X\to[0,\infty)$ by
	\begin{equation*}f_s(x)\,\df\, \begin{cases}
	0\,, & \text{if\ } L(x)\le \frac{s}{2}\,,\\[3pt]
	L(x)-\frac{s}{2}\,, & \text{if\ } L(x)> \frac{s}{2}\,.
	\end{cases}\end{equation*}
	We have
	\begin{equation}\label{EPT1.2A}
	\int_\X \bigl(f_s(x)\bigr)^p\, \uppi (\D x) \,\ge\, \Bigl(\frac{s}{2}\Bigr)^p\,
	\uppi\bigl(\{x\colon L(x) > s\}\bigr)\qquad \forall\,s>0\,.
	\end{equation}
	Since, by assumption,
	$\int_\X\bigl(L(x)\bigr)^{\vartheta+\varepsilon}\,
	\uppi(\D{x})=\infty$,
	there exists an increasing diverging sequence $\{s_n\}_{n\in\NN}\subset[0,\infty)$ such
	that
	\begin{equation}\label{EPT1.2B}
	\Bigl(\frac{s_n}{2}\Bigr)^p\,
	\uppi\bigl(\{x\colon L(x) > s_n\}\bigr)
	\,\ge\,2^ps_n^{p-\vartheta-\varepsilon-\iota}\,.
	\end{equation}
	Note also that
	$\bigr(f_s(x)\bigl)^p \le 2^{\theta-p}
	s^{p-\theta}\,\bigl(L(x)\bigr)^{\theta}\le\frac{2^{\theta-p}}{c}\,
	s^{p-\theta}\,\Lyap(x)$ for all $s>0$ and $x\in\X$. 
	This follows from the facts that $f_s(x)=0$ for $s>0$ and $x\in\X$  such that $L(x)\le s/2$,  
	$$0\,\le\,\frac{f_s(x)}{\frac{s}{2}}\,\le\,\frac{L(x)}{\frac{s}{2}}\qquad \forall\,(s,x) \in (0,\infty)\times\X\,,$$ and $\theta>p\ge1$.
	Thus, by the Foster-Lyapunov equation \cref{ET1.1A}
	(see \cite[Theorem~1.1]{Meyn-Tweedie-AdvAP-III-1993}), we obtain
	\begin{equation}\label{EPT1.2C}
	\int_\X \bigl(f_s(x)\bigr)^p\, \updelta_{x_0} P_t (\D x)
	\,\le\,\frac{2^{\theta-p}}{c}\, s^{p-\theta}
	\bigl(b\, t + \Lyap(x_0)\bigr)\qquad \forall\,s,t>0\,.
	\end{equation}
	Select a sequence $\{t_n\}_{n\in\NN}\subset[0,\infty)$ such that
	\begin{equation}\label{EPT1.2D}
	s_n^{\theta-\vartheta-\varepsilon-\iota} \,=\,
	\frac{2^{\theta-p}}{c}\, \bigl(b\, t_n+ \Lyap(x_0)\bigr)\,.
	\end{equation}
	Combining \cref{EPT1.2A,EPT1.2B,EPT1.2C,EPT1.2D} above we have
	\begin{equation*}
		\begin{aligned}
			\left(\int_\X \bigl(f_{s_n}(x)\bigr)^p\, \uppi (\D x)\right)^{\nicefrac{1}{p}}
			- \left(\int_\X \bigl(f_{s_n}(x)\bigr)^p\,
			\updelta_{x_0} P_{t_n} (\D x)\right)^{\nicefrac{1}{p}}
			&\,\ge\, (s_n)^{\frac{p-\vartheta-\varepsilon-\iota}{p}}\\
			&\,\ge\,
			\Bigl(\tfrac{2^{\theta-p}}{c}\,\bigl(b\, t_n
			+ \Lyap(x_0)\bigr)\Bigr)^{-\frac{\vartheta-p+\varepsilon+\iota}
				{(\theta-\vartheta-\varepsilon-\iota)p}}\quad \forall\,n\in\NN\,.
		\end{aligned}
	\end{equation*}
	The result then follows by \cite[Proposition~7.29]{Villani-Book-2009},
	which asserts that
	\begin{equation*}
		\Biggl|\left(\int_{\X} \bigl(f(x)\bigr)^p\, \upmu_1 (\D x)\right)^{\nicefrac{1}{p}}
		-\left(\int_{\X} \bigl(f(x)\bigr)^p\, \upmu_2 (\D x)\right)^{\nicefrac{1}{p}}\Biggr|
		\,\le\, \mathrm{Lip}\bigl(f\bigr) \,\sW_p(\upmu_1,\upmu_2)
	\end{equation*} for all $\upmu_1,\upmu_2\in\cP_p(\X)$ and Lipschitz $f:\X\to\RR$ with Lipschitz constant $\mathrm{Lip}\bigl(f\bigr).$
\end{proof}

For the proof of \cref{T1.3} we need two auxiliary results
given in \cref{L2.1,L2.2} below. First, recall that $\{X(t)\}_{t\ge0}$ is said to be conservative if $\Prob_x(X(t)\in\RR^n)=1$ for all $t\ge0$ and $x\in\RR^n$, and note that this is equivalent to 
\begin{equation*}\Prob_x\Bigl(\lim_{k\to\infty}\tau_k=\infty\Bigr)\,=\,1\qquad \forall\,x\in\RR^n\,,\end{equation*}
where $\tau_k\df\inf\{t\ge0\colon X_t\in\sB^c_k\}$ for $k\in\NN$ (here it is also essential that $\{X(t)\}_{t\ge0}$ has c\`{a}dl\`{a}g sample paths). Namely, for $t\ge0$ and $x\in\RR^n$   
it holds that $$\Prob_x(X(t)\in\RR^n)\,=\,\Prob_x\Bigl(\lim_{k\to\infty}\tau_k>t\Bigr)\,\ge\, \Prob_x\Bigl(\lim_{k\to\infty}\tau_k=\infty\Bigr)\,=\,\lim_{t\to\infty}\Prob_x\Bigl(\lim_{k\to\infty}\tau_k>t\Bigr)\,=\,\lim_{t\to\infty}\Prob_x(X(t)\in\RR^n)\,.$$

\begin{lem}\label{L2.1}
	Assume \hyperlink{LB}{\ttup{LB}} and \hyperlink{MP}{\ttup{MP}}.
	Then for any $x\in\RR^n$ and any nonnegative $f\in C^{\infty}(\RR^n)$ such that
the map $y\mapsto\int_{\sB^c}f(y+z)\,\upnu(y,\D z)$ is locally bounded,
	$\process{M_f}$ is a $\Prob_x$-local martingale
	{\upshape(}with respect to $\{\mathcal{F}_t\}_{t\ge0}${\upshape)}. 
\end{lem}

\begin{proof}
	For $k\in\NN$, let $\chi_k\in C_c^\infty(\RR^n)$ be such that
	$\Ind_{\sB_k}(x)\le\chi_k(x)\le\Ind_{\sB_{k+1}}(x)$
	and $\chi_k(x)\le\chi_{k+1}(x)$ for $x\in\RR^n$.
	Then, for any $x\in\RR^n$, $k,j\in\NN$ and $s,t\ge0$, $s\le t$,
	\cite[Theorem~2.2.13]{Ethier-Kurtz-Book-1986} implies that
	\begin{equation*}
		\Exp_x\bigl[M_{f\chi_k}(t\wedge\tau_j) \,|\,\mathcal{F}_s\bigr]
		\,=\,M_{f\chi_k}(s\wedge\tau_j)\,.
	\end{equation*}
	Next, by employing the monotone and dominated
	convergence theorems, we easily see that 
	\begin{equation*}\Exp_x\left[\biggl|\int_0^t\mathcal{L}f\bigl(X(s\wedge\tau_j)\bigr)\,\D s\biggr|\right]
	\,<\,\infty\qquad \forall\,(x,j)\in\RR^n\times\NN\,,\end{equation*}
	and
	\begin{align*}
		\Exp_x\bigl[f\bigl(X(t\wedge\tau_j)\bigr)\bigr]&\,=\,\lim_{k\to\infty}\,
		\Exp_x\bigl[f\bigl(X(t\wedge\tau_j)\bigr)\chi_k\bigl(X(t\wedge\tau_j)\bigr)\bigr]\\
		&\,=\,f(x)+\lim_{k\to\infty}\Exp_x\biggl[\int_0^t\mathcal{L}
		\bigl(f\bigl(X(s\wedge\tau_j)\bigr)\chi_k\bigl(X(s\wedge\tau_j)\bigr)\,\D s\biggr]\\
		&\,=\,f(x)+\Exp_x\biggl[\int_0^t\mathcal{L}f\bigl(X(s\wedge\tau_j)\bigr)\,\D s\biggr]
		\qquad \forall\,(x,j)\in\RR^n\times\NN\,.
	\end{align*}
	Hence, for each $x\in\RR^n$, $t\ge0$ and $j\in\NN$, $M_f(t\wedge\tau_j)$ is integrable.
	Also,
\begin{equation*}\lim_{k\to\infty}\Exp_x\bigl[M_{f\chi_k}(t\wedge\tau_j)\,|\,\mathcal{F}_s\bigr]
\,=\,\Exp_x\bigl[M_{f}(t\wedge\tau_j)\,|\,\mathcal{F}_s\bigr]\,,\quad\text{and\ }
\lim_{k\to\infty}M_{f\chi_k}(s\wedge\tau_j)\,=\,M_{f}(s\wedge\tau_j)\,,\end{equation*}
for all $x\in\RR^n$, $t\ge s\ge0$, and $j\in\NN$.
	The assertion now follows from the conservativeness of $\process{X}$.
\end{proof}

For $f\in C^1(\RR^n)$ we let 
\begin{equation*}\mathfrak{d} f(x;y)\,\df\,f(x+y)-f(x)-\langle y,\nabla f(x)\rangle\,,
\qquad x,y\in\RR^n\,,\end{equation*}
\begin{equation*}
\mathfrak{J}_{1,\upnu}[f](x)\,\df\,\int_{\RR^n}\mathfrak{d}_1 f(x;y)\, \upnu(x,\D{y})\,,
\quad\text{and\ }
\mathfrak{J}_\upnu[f](x)\,\df\,\int_{\RR^n}\mathfrak{d} f(x;y)\, \upnu(x,\D{y})\,,
\quad x\in\RR^n\,,
\end{equation*}
whenever the integrals are well defined. 

\begin{lem}\label{L2.2}
	Suppose that $\theta_\upnu>0$, and
	that \cref{ET1.3A} holds
	for some $\theta\in(0,\theta_\upnu]\cap\Theta_\upnu$.
	Then, we have the following:
	\begin{enumerate}
		\item[\ttup{i}] If $\theta\in(0,1)$, and 
		$f\in C^2(\RR^n)$ satisfies
		\begin{equation*}
		\sup_{x\in\sB^c}\,|x|^{-\theta}\max\bigl(\abs{f(x)},\abs{x}\,\abs{\nabla f(x)},
		\abs{x}^2\,\norm{\nabla^2f(x)}\bigr)\,<\,\infty\,,
		\end{equation*}
		then $\mathfrak{J}_{1,\upnu}[f](x)$
		vanishes at infinity.
		\item [\ttup{ii}]
		If $\theta\ge1$, and $f\in C^2(\RR^n)$ satisfies
		\begin{equation*}
			\sup_{x\in\sB^c}\,|x|^{1-\theta}\max\bigl(\abs{\nabla f(x)},
			\abs{x}\,\norm{\nabla^2f(x)}\bigr)\,<\,\infty\,,
		\end{equation*}
		then $\mathfrak{J}_\upnu[f](x)$
		vanishes at infinity when $\theta\in[1,2)$, and the map
		$x\mapsto(1+\abs{x})^{2-\theta}\,\mathfrak{J}_\upnu[f](x)$
		is bounded when $\theta\ge2$.
		
		\item [\ttup{iii}]
		If \cref{ET1.3B} holds
		for some $\theta>0$, then there exist
		$c>0$ and $r=r(\zeta)>0$, such that for any
		$\zeta\in\bigl(0,\frac{1}{2}\theta\norm{Q}^{-\nicefrac{1}{2}}\bigr)$ we have
		\begin{equation}\label{EL2.2A}
		\mathfrak{J}_\upnu\bigl[\widetilde\Lyap_{Q,\zeta}\bigr](x)\,\le\,
		c\,\zeta^{\nicefrac{3}{2}}\,\widetilde\Lyap_{Q,\zeta}(x)\qquad\forall\,x\in \sB_r^c\,.
		\end{equation}
	\end{enumerate}
\end{lem}

\begin{proof}
	The proof of parts (i) and (ii) follows as a straightforward adaptation of
	\cite[Lemma~~5.1]{Arapostathis-Pang-Sandric-2019}
	by setting
	\begin{align*}
		&C_0(\theta) \,\df\, \sup_{x\in\RR^n}\int_{\RR^n}
		\bigl(\abs{y}^{2}\wedge\abs{y}^{\theta}\bigr)\,\upnu(x,\D{y})\,,\qquad
		\widehat{C}_0(\theta) \,\df\, \sup_{x\in\RR^n}\int_{\sB}\abs{y}^{2}\,\upnu(x,\D{y})\,,
		\intertext{and}
		&\Breve{C}_0(r;\theta) \,\df\, \sup_{x\in\RR^n}\int_{\sB_r^c}\abs{y}^{\theta}
		\upnu(x,\D{y})\,,\qquad r>0\,.
	\end{align*}
	
	To prove part (iii), we use the identity
	\begin{equation}\label{PL2.2A}
	\int_{\RR^n}\mathfrak{d}\widetilde\Lyap_{Q,\zeta}(x;y)\, \upnu(x,\D{y})
	\,=\,\int_{\RR^n}\int_0^1(1-t)
	\bigl\langle y,\nabla^2\widetilde\Lyap_{Q,\zeta}(x+ty)y\bigr\rangle\, \D{t}\,\upnu(x,\D{y})
	\end{equation}
	Consider the set
	\begin{equation*}A_x\,\df\,\bigl\{(t,y)\in[0,1]\times\RR^n\colon\abs{x+ty}_Q
	\le \tfrac{1}{2}\abs{x}_Q\bigr\}\,,\qquad x\in\RR^n\,.\end{equation*}
	On this set we have the bound
	\begin{equation}\label{PL2.2B}
	\babs{\bigl\langle y,\nabla^2\widetilde\Lyap_{Q,\zeta}(x+ty)y\bigr\rangle}
	\,\le\, \bar c (\zeta  + \zeta^2)\abs{y}^2\,
	\E^{\zeta\norm{Q}^{\nicefrac{1}{2}}\abs{ty}}\,
	\widetilde\Lyap_{Q,\zeta}(x)
	\end{equation}
	for some  $\bar c\ge1$.
	Since $\zeta\norm{Q}^{\nicefrac{1}{2}}<\theta$,
	and $\abs{y}_Q\ge\abs{ty}_Q\ge\frac{1}{2}\abs{x}_Q$ on the set $A_x$,
	there exists  $\rho=\rho(\zeta)\ge1$ such that
	\begin{equation}\label{PL2.2C}
	\zeta^{-\nicefrac{1}{2}}(1  + \zeta)\abs{y}^2\,
	\E^{\zeta\norm{Q}^{\nicefrac{1}{2}}\abs{ty}}
	\,\le\, \E^{\theta\abs{y}}\end{equation}
	for all $x\in\sB^c_{2\rho}$ and $(t,y)\in A_x$.
	Hence, using \cref{PL2.2B,PL2.2C} and Fubini's theorem, we have
	\begin{equation}\label{PL2.2D}
	\iint_{A_x}(1-t)
	\bigl\langle y,\nabla^2\widetilde\Lyap_{Q,\zeta}(x+ty)y\bigr\rangle\,
	\D{t}\,\upnu(x,\D{y})
	\,\le\,2\,\bar c\, \zeta^{\nicefrac{3}{2}}
	\biggl(\int_{ \sB_{\rho}^c}\E^{\theta\abs{y}}\,\upnu(x,\D y)\biggr)
	\widetilde\Lyap_{Q,\zeta}(x)
	\end{equation} for all $x\in\sB^c_{2\rho}$.
	Next, since $\abs{x+ty}_Q > \frac{1}{2}\abs{x}_Q$ on the set $A^c_x$,
	we have a bound of the form
	\begin{equation}\label{PL2.2E}
	\bigl\langle y,\nabla^2\widetilde\Lyap_{Q,\zeta}(x+ty)y\bigr\rangle
	\,\le\, \bar c
	\biggl(\zeta^2 + \frac{\zeta}{\abs{x}}\biggr)\abs{y}^2\,
	\E^{\zeta\norm{Q}^{\nicefrac{1}{2}}\abs{y}}\,
	\widetilde\Lyap_{Q,\zeta}(x)
	\end{equation}
	for all $x\in\sB^c$ and $(t,y)\in A_x^c$,
	where, without loss of generality, we use the same constant $\bar c$
	as in \cref{PL2.2B}.
	Since $\theta> 2\zeta\sqrt{\norm{Q}}$, 
	it is clear that there exists $\tilde c>0$, independent of $\zeta$,
	such that
	\begin{equation}\label{PL2.2F}
	\biggl(\zeta^2 + \frac{\zeta}{\abs{x}}\biggr)\abs{y}^2\,
	\E^{\zeta\norm{Q}^{\nicefrac{1}{2}}\abs{y}}
	\,\le\, \tilde c\,\zeta^{\nicefrac{3}{2}}\,\bigl(
	\abs{y}^2\Ind_{\sB}(y)+\E^{\theta\abs{y}}\Ind_{\sB^c}(y)\bigr)
	\qquad\forall(x,y)\in \sB_{1/\zeta}^c\times\RR^n\,.
	\end{equation}
	Thus, by \cref{ET1.3B,PL2.2E,PL2.2F}, there exists 
	$\hat c>0$ such that
	\begin{equation}\label{PL2.2G}
	\iint_{A_x^c}(1-t)
	\bigl\langle y,\nabla^2\widetilde\Lyap_{Q,\zeta}(x+ty)y\bigr\rangle\,
	\D{t}\,\upnu(x,\D{y})
	\,\le\,\hat c\, \zeta^{\nicefrac{3}{2}}\,\widetilde\Lyap_{Q,\zeta}(x)
	\qquad \forall\, x\in\sB^c_{1/\zeta}\,.
	\end{equation}
	The estimate in \cref{EL2.2A} follows from \cref{ET1.3B,PL2.2A,PL2.2D,PL2.2G}.
	This completes the proof.
\end{proof}

We next prove \cref{T1.3}.

\begin{proof}[Proof of \cref{T1.3}]
	In cases \ttup{i} and \ttup{ii}, we take $\Lyap(x)=\Lyap_{Q,\theta}(x)+1$,
	while in case \ttup{iii} we use $\Lyap(x)=\widetilde\Lyap_{Q,\zeta}(x)$
	with $\zeta>0$ sufficiently small.
	Then, in view of \cref{L2.2} it is straightforward to see that there exist
	constants $\Bar c>0$, $\Tilde c>0$, and $r>0$, such that
	\begin{equation*}\mathcal{L}\Lyap(x)\,\le\,\Bar c\,\Ind_{\Bar{\sB}_r}(x)-\Tilde c\,
	\bigl(\Lyap(x)\bigr)^{\nicefrac{(\theta-2+\vartheta)}{\theta}}
	\qquad \forall\,x\in\RR^n\,,\end{equation*}
	in case \ttup{i}, and
	\begin{equation*}\mathcal{L}\Lyap(x)\,\le\,\Bar c\,\Ind_{\Bar{\sB}_r}(x)-\Tilde c\,\Lyap(x)
	\qquad \forall\,x\in\RR^n\,,\end{equation*}
	in cases \ttup{ii} and \ttup{iii}. 
	Observe that the above relations, together with
	\cite[Theorem~2.1]{Meyn-Tweedie-AdvAP-III-1993} and \cref{L2.1},
	imply that $\process{X}$ is conservative.
	Finally, according to \cref{L2.1} and \cite[Theorem~3.4]{Douc-Fort-Guilin-2009}
	the process $\process{X}$ satisfies \cref{ET1.1A} with
	$\phi(t)=t^{\nicefrac{(\theta-2+\vartheta)}{\theta}}$ in case \ttup{i},
	and $\phi(t)=t$ in cases \ttup{ii} and \ttup{iii}
	(for some $b>0$ and closed petite set $C$).
\end{proof}

The proof of \cref{T1.4} is based on the following lemma.

\begin{lem}\label{L2.3}
	Let $\process{X}$ be an It\^{o} process with locally bounded
	coefficients $b(x)$ and $a(x)$
	and satisfying the linear growth condition in \cref{LGC},
	and $\upnu(x,\D y)$ such that $\theta_\upnu>0$.
	Then, for any $\theta\in[0,\theta_\upnu]\cap\Theta_\upnu$,
	there exists a constant $c>0$ such that 
	\begin{equation*}
		\Exp_{x}\bigl[|X(t)|^\theta\bigr] \,\le\,
		\bigl(1+|x|^\theta\bigr)\,\E^{ct} \qquad \forall\,(t,x) \in [0,\infty)\times\RR^n\,.
	\end{equation*}
\end{lem}

\begin{proof}
	Let $\varphi\in C^2(\RR^n)$ be such that $\varphi(x)\ge0$ and $\varphi(x)\le|x|^\theta$ for $x\in\RR^n$,
	and $\varphi(x)=|x|^\theta$ for $x\in\sB^c$.
	Further, for $k\in\NN$, let $\varphi_k\in C^2_b(\RR^n)$ be such that
	$\varphi_k(x)\ge0$, $\varphi_k(x)=\varphi|_{\sB_{k+1}}(x)$, and
	$\varphi_k(x)\nearrow \varphi(x)$, as $k\to\infty$, for every $x\in\RR^n$. 
	Then, according to It\^o's formula and the conservativeness of $\process{X}$
	we have 
	\begin{align*}
		\Exp_x\bigl[\varphi_k\bigl(X(t\wedge\tau_k)\bigr)\bigr]&\,=\, \varphi_k(x)
		+\Exp_x\left[\int_0^{t\wedge\tau_k}\mathcal{L}\varphi_k\bigl(X(s)\bigr)\,\D s\right]\\
		&\,\le\,
		\varphi_k(x)+c_k(t\wedge\tau_k)
		+c_k\,\Exp_x\left[\int_0^{t\wedge\tau_k}\varphi_k\bigl(X(s)\bigr)\,\D s\right]\\
		&\,\le\, \varphi_k(x)+c_kt+c_k\int_0^{t}
		\Exp_x\left[\varphi_k\bigl(X(s\wedge\tau_k)\bigr)\right]\,\D s
	\end{align*}
	for all $k\in\NN$, $t\ge0$, and $x\in\RR^n$,
	where the constants $c_k>0$ depend on $\theta$, $b(x)$, $a(x)$, and
	the quantities
	\begin{equation*}
		\sup_{x\in\RR^n}\,\int_{\RR^n}
		\left(|y|^2\Ind_{\sB}(y)+|y|^\theta\Ind_{\sB^c}(y)\right)\upnu(x,\D y)
		\quad \text{and\ }
		\sup_{x\in \sB_r}\,\Bigl(\babs{\varphi_k(x)}+\babs{\nabla\varphi_k(x)}
		+\babs{\nabla^2\varphi_k(x)}\Bigr)\,,
	\end{equation*}
	for $r>0$ large enough.
	Clearly, the functions $\varphi_k(x)$ can be chosen such that
	$c\df \sup_{k\in\NN}c_k<\infty$.
	Now, since the function $t\mapsto\Exp_x\bigl[\varphi_k(X\bigl(t\wedge\tau_k)\bigr)\bigr]$
	is bounded and c\`adl\`ag, Gronwall's lemma implies that
	\begin{equation*}
		\Exp_x\bigl[\varphi_k\bigl(X(t\wedge\tau_k)\bigr)\bigr]\,\le\,
		\bigl(1+\varphi_k(x)\bigr)\, \E^{c t} -1
	\end{equation*}  for all $k\in\NN$, $t\ge0$, and $x\in\RR^n$.
	By letting $k\to\infty$,  Fatou's lemma and the conservativeness of
	$\process{X}$ imply that
	\begin{equation*}
		\Exp_x\bigl[\varphi\bigl(X(t)\bigr)\bigr]\,\le\,
		\bigl(1+\varphi(x)\bigr)\, \E^{c t} -1\qquad \forall\,(t,x) \in [0,\infty)\times\RR^n\,.
	\end{equation*}
	Finally, we have that
\begin{equation*}
		\Exp_x\bigl[|X(t)|^\theta\bigr]\,\le\,\Exp_x\bigl[\varphi\bigl(X(t)\bigr)\bigr]+1
		\,\le\,\bigl(1+\varphi(x)\bigr)\,\E^{c t}
		\,\le\,(1+|x|^\theta)\,\E^{c t}\quad \forall\,(t,x) \in [0,\infty)\times\RR^n\,.
\end{equation*}
	This completes the proof.
\end{proof}

We next prove \cref{T1.4}.

\begin{proof}[Proof of \cref{T1.4}]
	For $p\in[2,\theta_\upnu]\cap\Theta_\upnu$, define
		$\Lyap_{p}(x) \df \abs{x}_Q^{p}$, $x\in\RR^n$,
	and
	\begin{align*}
		\Tilde{\Lg} f(x;z) \,\df\,& \sum^{n}_{i=1} \varDelta_{z}\Tilde b_{i}(x)
		\frac{\partial f}{\partial x_{i}}(z)
		+ \frac{1}{2} \sum^{n}_{i,j=1}\Tilde{a}_{ij}(x;z)
		\frac{\partial^{2} f}{\partial x_{i}\partial x_{j}}(z)\\
		&+\int_{\RR^n}\left(f(z+y)-f(y)
		-\sum_{i=1}^n y_i\frac{\partial f}{\partial x_{i}}(z)\right)
		\varDelta_{z}\upnu(x,\D y)\,,\qquad x,z\in\RR^n\,.
	\end{align*}
	Calculating $\Tilde{\Lg}\Lyap_{p}(x;z)$, using
	\cref{ET1.4A}, we obtain
	\begin{align*}
		\Tilde{\Lg}\Lyap_{p}(x;z)
		&\,=\, \frac{p}{2}\,\abs{z}_Q^{p-2}\left(2\,\bigl\langle \varDelta_{z}
		\Tilde b(x),Qz\bigr\rangle+\trace\,\bigl(\Tilde{a}(x;z) Q\bigr)\right)
		+\frac{p(p-2)}{2}\,\abs{z}_Q^{p-4}\abs{\varDelta_{z}\sigma'(x) Qz}^{2}\\
		&\mspace{20mu} +\int_{\RR^n}\int_0^1(1-t)
		\bigl\langle y,\nabla^2\Lyap_{\varepsilon,p}(z+ty)y\bigr\rangle\,\D t\,
		\varDelta_{z}\upnu(x,\D y)\\
		&\,=\,\frac{p}{2}\,\abs{z}_Q^{p-2}\left(2\,\bigl\langle \varDelta_{z}
		\Tilde b(x),Qz\bigr\rangle+\trace\,\bigl(\Tilde{a}(x;z) Q\bigr)\right)
		+\frac{p(p-2)}{2}\,\abs{z}_Q^{p-4}\abs{\varDelta_{z}\sigma'(x) Qz}^{2}\\
		&\mspace{20mu} +\int_{\RR^n}\int_0^1(1-t)\left(p\, \abs{z+ty}_Q^{p-2}
		\abs{y}^2_Q+p(p-2)\abs{z+ty}_Q^{p-4}\abs{y'Q(z+ty)}^2\right)\D t\,
		\varDelta_{z}\upnu(x,\D y)\\
		&\,\le\, \frac{p}{2}\,\abs{z}_Q^{p-2}\left(2\,\bigl\langle \varDelta_{z}
		\Tilde b(x),Qz\bigr\rangle+\trace\,\bigl(\Tilde{a}(x;z) Q\bigr)+(p-2)\,
		\abs{\sqrt{Q}\,\varDelta_{z}\sigma(x)}^{2}\right)\\
		&\mspace{20mu} +p\bigl(1+(p-2)\norm{Q^{-1}}\bigr)\int_{\RR^n}
		\abs{y}_Q^2\int_0^1(1-t )\abs{z+ty}_Q^{p-2}\,\D t\, \varDelta_{z}\upnu(x,\D y)\\
		&\,\le\, \frac{p}{2}\,\abs{z}_Q^{p-2}\left(2\,\bigl\langle \varDelta_{z}
		\Tilde b(x),Qz\bigr\rangle+\trace\,\bigl(\Tilde{a}(x;z) Q\bigr)+(p-2)\,
		\abs{\sqrt{Q}\,\varDelta_{z}\sigma(x)}^{2}\right)\\
		&\mspace{20mu} +p2^{p-4}\bigl(1+(p-2)\norm{Q^{-1}}\bigr)\abs{z}_Q^{p-2}
		\int_{\RR^n}\abs{y}_Q^2\, \varDelta_{z}\upnu(x,\D y)\\
		&\mspace{20mu} +\frac{2^{p-3}}{p-1}\bigl(1+(p-2)\norm{Q^{-1}}\bigr)
		\int_{\RR^n}\abs{y}_Q^{p}\, \varDelta_{z}\upnu(x,\D y)\\
		&\,=\, \frac{p}{2}\,\abs{z}_Q^{p-2}\Biggl(2\,\bigl\langle \varDelta_{z}
		\Tilde b(x),Qz\bigr\rangle+\trace\,\bigl(\Tilde{a}(x;z) Q\bigr)+(p-2)\,
		\abs{\sqrt{Q}\,\varDelta_{z}\sigma(x)}^{2}\\
		&\mspace{120mu}+2^{p-3}\bigl(1+(p-2)\norm{Q^{-1}}\bigr)
		\int_{\RR^n}\abs{y}_Q^2\, \varDelta_{z}\upnu(x,\D y)\\
		&\mspace{120mu} +\frac{2^{p-2}}{p(p-1)}\bigl(1+(p-2)\norm{Q^{-1}}\bigr)
		\abs{z}_Q^{2-p}\int_{\RR^n}\abs{y}_Q^{p}\, \varDelta_{z}\upnu(x,\D y)\Biggr)
			\end{align*}
		\begin{align*}
		&\,=\,\frac{p}{2}\,\abs{z}_Q^{p-2}\Biggl(2\,\bigl\langle \varDelta_{z}
		\Tilde b(x),Qz\bigr\rangle+\trace\,\bigl(\Tilde{a}(x;z) Q\bigr)+(p-2)\,
		\abs{\sqrt{Q}\,\varDelta_{z}\sigma(x)}^{2}\\
		&\mspace{120mu} +2^{p-3}\bigl(1+(p-2)\norm{Q^{-1}}\bigr)\int_{\RR}
		\abs{k(x+z,v)-k(x,v)}_Q^2\, \upnu(\D v)\\
		&\mspace{120mu} +\frac{2^{p-2}}{p(p-1)}\bigl(1+(p-2)\norm{Q^{-1}}\bigr)
		\abs{z}_Q^{2-p}\int_{\RR}\abs{k(x+z,v)-k(x,v)}_Q^{p}\, \upnu(\D v)\Biggr)\\
		&\,\le\, -c(p)\,\Lyap_{p}(z)
	\end{align*}
for all $x,z\in\RR^n$.
	Next, for $x,z\in\RR^n$, let $\tau \df \inf\{t\ge 0\colon\, X^{x+z}(t) = X^{x}(t)\}$
	(possibly $+\infty$), where $\process{X^x}$ denotes the solution to \cref{SDE} 
	with $X^x(0)=x$ for $x\in\RR^n$.
	By It\^o's formula and the conservativeness of $\process{X}$ we obtain
	\begin{align*}
		\Exp\bigl[\Lyap_{p}\bigl(X^{x+z}(t\wedge\tau\wedge\tau_k)&
		-X^{x}(t\wedge\tau\wedge\tau_k)\bigr)\bigr] - \Lyap_{p}(z)\\
		&\,=\,
		\Exp\biggl[ \int_{0}^{t\wedge\tau\wedge\tau_k}
		\Tilde{\Lg}\Lyap_{p}\bigl(X^{x}(s);X^{x+z}(s)
		- X^{z}(s)\bigr) \,\D{s}\biggr]\\
		&\,=\,
		\Exp\biggl[ \int_{0}^{t}
		\Tilde{\Lg}\Lyap_{p}\bigl(X^{x}(s\wedge\tau\wedge\tau_k);X^{x+z}(s\wedge\tau)
		- X^{z}(s\wedge\tau\wedge\tau_k)\bigr) \,\D{s}\biggr]\\
		&\,=\,
		\int_{0}^{t}
		\Exp\Bigl[
		\Tilde{\Lg}\Lyap_{p}\bigl(X^{x}(s\wedge\tau\wedge\tau_k);X^{x+z}(s\wedge\tau\wedge\tau_k)
		- X^{z}(s\wedge\tau\wedge\tau_k)\bigr)\Bigr] \,\D{s}
	\end{align*} for all $t\ge0$ and $k\in\NN$,
	since, for $t \ge \tau$, $X^{x+z}(t) = X^{x}(t)$ a.s.\ by the pathwise
	uniqueness of the solution to \cref{SDE}.
	From this and \cref{L2.3} we conclude that the function
	$t\mapsto \Exp\bigl[\Lyap_{p}\bigl(X^{x+z}(t\wedge\tau\wedge\tau_k)
	-X^{x}(t\wedge\tau\wedge\tau_k)\bigr)\bigr]$ is differentiable a.e. on $(0,\infty)$.
	Note that $|\Tilde{\Lg}\Lyap_{p}\bigl(x;z\bigr)|\le c|z|^p$
	for some $c>0$ and all $x,z\in\RR^n$,
	We conclude now that
	\begin{align*}
		\frac{\D}{\D{t}}\Exp\bigl[\Lyap_{p}
\bigl(X^{x+z}(t\wedge\tau\wedge\tau_k) &-X^{x}(t\wedge\tau\wedge\tau_k)\bigr)\bigr]\\
		&\,=\,\Exp\bigl[
		\Tilde{\Lg}\Lyap_{p}\bigl(X^{x}(s\wedge\tau\wedge\tau_k);X^{x+z}(t\wedge\tau\wedge\tau_k)
		- X^{z}(t\wedge\tau\wedge\tau_k)\bigr)\bigr]\\
		&\,\le\,-c(p)\,\Exp\bigl[\Lyap_{p}
		\bigl(X^{x+z}(t\wedge\tau\wedge\tau_k)-X^{x}(t\wedge\tau\wedge\tau_k)\bigr)\bigr]
		\qquad \text{a.e. on}\ (0,\infty)
	\end{align*} for all $k\in\NN$.
	Thus by Gronwall's lemma, it follows that
\begin{equation*}
\Exp \bigl[\Lyap_{p}\bigl(X^{x+z}(t\wedge\tau_k) - X^{x}(t\wedge\tau_k)\bigr)\bigr]
\,=\,\Exp\bigl[\Lyap_{p}
\bigl(X^{x+z}(t\wedge\tau\wedge\tau_k)-X^{x}(t\wedge\tau\wedge\tau_k)\bigr)\bigr]
\,\le\,
\Lyap_{p}(z)\,\E^{-c(p) t}\,,
\end{equation*}
	and Fatou's lemma implies that
\begin{equation*}
\Exp \bigl[\Lyap_{p}\bigl(X^{x+z}(t) - X^{x}(t)\bigr)\bigr]
	\,\le\,
	\Lyap_{p}(z)\,\E^{-c(p) t}
\end{equation*}
 for all $t\ge0$ and $x,z\in\RR^n$.
	Next, from the bound
	$\underline\lambda_Q\abs{z}^{2}\le \abs{z}^2_Q\le \Bar\lambda_Q\abs{z}^{2}$ we obtain
	\begin{align*}
		\Exp\bigl[|X^{x+z}(t) - X^{x}(t)|^p\bigr] &\,\le\,
		\bigl(\underline\lambda_Q)^{-\nicefrac{p}{2}}\,
		\Exp\bigl[\abs{X^{x+z}(t) - X^{x}(t)}^p_Q\bigr]\\
		&\,\le\, \bigl(\underline\lambda_Q)^{-\nicefrac{p}{2}}\,
		\bigl(\overline\lambda_Q)^{\nicefrac{p}{2}}\,
		|z|^{p}\,\E^{{-c(p)t}}
	\end{align*} for all $t\ge0$ and $x,z\in\RR^n$,
	thus establishing \cref{ET1.4B}.
	
	Finally, in order to establish \cref{ET1.4C},
	we follow the idea from \cite[Proof of Corollary~1.8]{Luo-16} or
	\cite[Proof of Theorem~2.1]{Komorowski-Walczuk-2012}.
	Observe first that, according to \cref{L2.3},
	for any $\upmu\in\cP_p(\RR^{n})$, $\upmu P_t\in\cP_p(\RR^{n})$ for all $t\geq0$.
	Next, let $\upmu_1,\upmu_2\in\cP_p(\RR^n)$ be arbitrary.
	According to \cref{ET1.4B}, we have
	\begin{equation*}
		\sW_p(\upmu_1 P_t,\upmu_2 P_t) \,\le\,
		\left(\frac{\overline{\lambda}_Q}{\underline{\lambda}_Q}\right)^{\nicefrac{1}{2}}
		\sW_p(\upmu_1,\upmu_2)\E^{-\frac{c(p)t}{p}}\qquad \forall\,t\geq0\,.
	\end{equation*}
	Fix $t_0\geq0$ such that
	\begin{equation*}
		\left(\frac{\overline{\lambda}_Q}{\underline{\lambda}_Q}\right)^{\nicefrac{1}{2}}
		\E^{-\frac{c(p)t_0}{p}}<1\,.
	\end{equation*}
	Then, the mapping $\upmu\mapsto\upmu P_{t_0}$
	is a contraction on $\cP_p(\RR^{n})$.
	Thus, since $(\cP_p(\RR^{n}),\sW_p)$ is a complete metric space, the Banach fixed
	point theorem entails that there exists a unique $\uppi_{0}\in\cP_p(\RR^{n})$
	such that $\uppi_{0}P_{t_0}(\D x)=\uppi_{0}(\D x)$.
	By defining $\uppi(\D x)\df t_0^{-1}\int_0^{t_0}\uppi_{0}P_s(\D x)\,\D s$, we can
	easily see that $\uppi P_{t}(\D x)=\uppi(\D x)$ for all $t\geq0$, i.e. $\uppi(\D x)$
	is an invariant probability measure for $\process{X}$.
	By employing \cref{L2.1} again, we also see that
	$\uppi\in\mathcal{P}_p(\RR^{n})$.
	Finally, for any $\upmu\in\mathcal{P}_p(\RR^{n})$ we have 
	\begin{equation*}
		\sW_p(\upmu P_t,\uppi) \,=\, \sW_p(\upmu P_t,\uppi P_t) \,\le\,
		\left(\frac{\overline{\lambda}_Q}{\underline{\lambda}_Q}\right)^{\nicefrac{1}{2}}
		\sW_p(\upmu,\uppi)\E^{-\frac{c(p)t}{p}}\qquad \forall\,t\geq 0\,,
	\end{equation*}
	which also proves uniqueness of $\uppi(\D x)$.

	To prove the second assertion, we adapt the proof of \cite[Lemma~7.3.4]{book},
	where an analogous result is shown for $p=1$.
	Define
	\begin{equation*}
		V_{\varepsilon,p}(x) \,\df\, \frac{\abs{x}_Q^{p+1}}
		{\bigl(\varepsilon +\abs{x}^2_Q\bigr)^{\nicefrac{1}{2}}}\,,
		\qquad \varepsilon>0\,,\quad x\in\RR^n\,,
	\end{equation*}
	and observe that in this case $\Tilde{\Lg}$ reduces to
	\begin{equation*}
		\Tilde{\Lg} f(x;z) \,=\, \sum^{d}_{i=1} \varDelta_{z}b^{i}(x)
		\frac{\partial f}{\partial x_{i}}(z)\qquad \forall\, x,z\in\RR^n\,.
	\end{equation*}
	Calculating $\Tilde{\Lg}V_{\varepsilon,p}(x;z)$, using
	\cref{ET1.4A}, we obtain
	\begin{align*}
		\Tilde{\Lg}\Lyap_{\varepsilon,p}(x;z)
		&\,=\, 
		\frac{\bigl(\varepsilon(p+1) + p\abs{z}^2_Q\bigr)\,\abs{z}_Q^{p-1}}
		{\bigl(\varepsilon + \abs{z}^2_Q\bigr)^{\nicefrac{3}{2}}}
		\, \bigl\langle Qz,\varDelta_{z}b(x)\bigr\rangle\\
		&\,\le\, -c(p)\, \frac{\varepsilon\frac{p+1}{p} + \abs{z}^2_Q}
		{\varepsilon + \abs{z}^2_Q}\,\Lyap_{\varepsilon,p}(z)
		\\
		&\,\le\, -c(p)\, \Lyap_{\varepsilon,p}(z)\qquad \forall\,x,z\in\RR^n\,.
	\end{align*} As before, 
	by It\^o's formula and the conservativeness of $\process{X}$, combined with the fact
	that the L\'evy noise does not depend on the state, we obtain
	\begin{align*}
		\Exp\bigl[\Lyap_{\varepsilon,p}\bigl(X^{x+z}(t\wedge\tau\wedge\tau_k)
		&-X^{x}(t\wedge\tau\wedge\tau_k)\bigr)\bigr] - \Lyap_{\varepsilon,p}(z)\\
		&\,=\,
		\int_{0}^{t} \Exp\Bigl[ \Tilde{\Lg}\Lyap_{\varepsilon,p}
		\bigl(X^{x}(s\wedge\tau\wedge\tau_k);X^{x+z}(s\wedge\tau\wedge\tau_k)
		- X^{z}(s\wedge\tau\wedge\tau_k)\bigr)\Bigr] \,\D{s}
	\end{align*}
	for all $t\ge0$ and $k\in\NN$,
	and 
	\begin{align*}
		\frac{\D}{\D{t}}\Exp\bigl[\Lyap_{\varepsilon,p}
\bigl(X^{x+z}(t\wedge\tau\wedge\tau_k)&-X^{x}(t\wedge\tau\wedge\tau_k)\bigr)\bigr]\\
		&\,=\,\Exp\bigl[
		\Tilde{\Lg}\Lyap_{\varepsilon,p}
		\bigl(X^{x}(t\wedge\tau\wedge\tau_k);X^{x+z}(t\wedge\tau\wedge\tau_k)
		- X^{z}(t\wedge\tau\wedge\tau_k)\bigr)\bigr]\\
		&\,\le\,-c(p)\,\Exp\bigl[\Lyap_{\varepsilon,p}
		\bigl(X^{x+z}(t\wedge\tau\wedge\tau_k)-X^{x}(t\wedge\tau\wedge\tau_k)\bigr)\bigr]
		\qquad \text{a.e. on} \ (0,\infty)
	\end{align*} for all $k\in\NN$.
	Thus by Gronwall's and Fatou's lemmas it follows that
	\begin{equation*}
		\Exp \bigl[\Lyap_{\varepsilon,p}\bigl(X^{x+z}(t) - X^{x}(t)\bigr)\bigr]
		\,\le\,
		\Lyap_{\varepsilon,p}(z)\,\E^{-c_p t}\end{equation*} for all $t\ge0$ and $x,z\in\RR^n$.
	Taking limits as $\varepsilon\to0$, and using
	monotone convergence, the assertion follows.
\end{proof}

In what follows we give an alternative proof of \cref{T1.4} in the case when
$\sigma(x)\equiv\sigma$ and $\upnu(x,\D y)\equiv\upnu(\D y)$.
Let $\Bar X(t)\df Q^{\nicefrac{1}{2}}X(t)$ for $t\ge0$. Clearly, $\process{\Bar X}$
is again an It\^{o} process which satisfies \begin{equation*}
	\Bar X(t) \,=\,
	x+Q^{\nicefrac{1}{2}}\int_0^tb\bigl(Q^{-\nicefrac{1}{2}}\Bar X(s)\bigr)\,\D s
	+Q^{\nicefrac{1}{2}}\sigma\,B(t)
	+Q^{\nicefrac{1}{2}} L(t)\qquad \forall\, t\ge0\,,
\end{equation*}
where $\process{L}$ is an $n$-dimensional pure-jump and zero-drift L\'evy process
determined by $\upnu(\D y)$.
The corresponding transition probability satisfies
\begin{align*}
	\Bar p(t,x,\D y)&\,=\, \Bar\Prob^x(\Bar X(t)\in\D y)\\
	&\,=\,\Prob^{Q^{-\nicefrac{1}{2}}x}(X(t)\in Q^{-\nicefrac{1}{2}}\D y)
	\,=\,p(t,Q^{-\nicefrac{1}{2}}x,Q^{-\nicefrac{1}{2}}\D y)
	\quad \forall\, (t,x)\in[0,\infty)\times\RR^n\,.
\end{align*}
Thus, we have
\begin{equation}
\begin{aligned}\label{EPT1.4A}
\langle \varDelta_z\Bar b(x),Qz\rangle &\,=\,
\langle Q^{\nicefrac{1}{2}}\varDelta_{Q^{-\nicefrac{1}{2}}z}
b(Q^{-\nicefrac{1}{2}}x),Qz\rangle\\\
&\,=\, \langle \varDelta_{Q^{-\nicefrac{1}{2}}z}
b(Q^{-\nicefrac{1}{2}}x),Q^{\nicefrac{1}{2}}z\rangle \\
&\,\le\, -\frac{c(p)}{p}\abs{Q^{-\nicefrac{1}{2}}z}_Q
\,=\, -\frac{c(p)}{p}\abs{z}_Q\qquad \forall\,x,z\in\RR^n\,.
\end{aligned}
\end{equation}
Now, in \cite{Bolley-Gentil-Guillin-2012} it has been shown that 
\cref{EPT1.4A} implies that
\begin{equation*}\sW_p(\updelta_x\Bar P_t,\updelta_y \Bar P_t) \,\le\,
	|x-y|\,\E^{-\frac{c(p)t}{p}}\end{equation*}
for all $t\ge0$ and $x,y\in\RR^n$.
Finally we get
\begin{equation*}
\begin{aligned}
\sW_p(\updelta_x P_t,\updelta_y P_t)& \,=\,
\sW_p\bigl(\Bar p(t,Q^{\nicefrac{1}{2}}x,Q^{\nicefrac{1}{2}}\D z),
\Bar p(t,Q^{\nicefrac{1}{2}}y,Q^{\nicefrac{1}{2}}\D z)\bigr)\\
&\,\le\, \bigl(\underline{\lambda}_Q\bigr)^{-\nicefrac{1}{2}}|
Q^{\nicefrac{1}{2}}(x-y)|\,\E^{-\frac{c(p)t}{p}}\\
&\,\le\, 
\left(\frac{\overline{\lambda}_Q}{\underline{\lambda}_Q}\right)^{\nicefrac{1}{2}}|x-y|\,
\E^{-\frac{c(p)t}{p}}
\end{aligned}
\end{equation*}
for all $t\ge0$ and $x,y\in\RR^n$,
which is \cref{ET1.4B}.

Lastly, we prove \cref{P1.1}.

\begin{proof}[Proof of \cref{P1.1}]
	According to \cite[Theorem~4.1]{Villani-Book-2009},
	for each $s\in[0,\infty)$ there exists $\Pi_s\in\mathcal{C}(\updelta_x P_s,\uppi)$
	such that $\sW_{p}(\updelta_x P^\psi_s,\uppi)=
	\int_{\X\times \X}\mathsf{d}(y,z)\,\Pi_s(\D y,\D z)$.
	Now, we have that
	\begin{align*}
		\sW^p_{p}(\updelta_xP_t^\psi,\uppi)&\,=\,
		\inf_{\Pi\in\mathcal{C}(\updelta_xP_t^\psi,\uppi)}\,
		\int_{\X\times\X}\bigl(\mathsf{d}(y,z)\bigr)^p\,\Pi(\D y,\D z)\\
		&\,\le\,
		\int_{\X\times\X}\bigl(\mathsf{d}(y,z)\bigr)^p
		\int_{[0,\infty)}\Pi_s(\D y,\D z)\,\upmu_t(\D s)\\
		&\,\le
		\,\int_{[0,\infty)} \sW^p_{p}(\updelta_xP_s,\uppi)\,\upmu_t(\D s)\\
		&\,\le\, \bigl(c(x)\bigr)^p\int_{[0,\infty)} \bigl(r^p(s)\bigr)^p\,\upmu_t(\D s)
		\,=\,\bigl(c(x)\bigr)^p\,\mathbb{E}\bigl[\bigl(r(S(t))\bigr)^p\bigr]\qquad
		\forall\, (t,x)\in[0,\infty)\times\X\,,
	\end{align*}
	which completes the proof.
\end{proof}

\section{Examples}\label{S3}

In this section, we consider applications of the main results to several classes
of Markov processes, including Langevin tempered diffusion processes, 
Ornstein-Uhlenbeck processes with jumps, piecewise Ornstein-Uhlenbeck processes
with jumps under constant and stationary Markov controls, state-space models
and backward recurrence time chains.
Further examples can be found in
\cite{Douc-Fort-Guilin-2009,DFMS-04,Fort-03,Fort-Roberts-2005,TT-94}.

\subsection{Langevin tempered diffusion processes}
We first consider a class of Langevin tempered diffusion processes.
Let $\alpha\in(0,1/n)$, and let $\pi\in C^2(\RR^n)$ be strictly positive,
$\pi(x)=c\,\abs{x}^{-\nicefrac{1}{\alpha}}$ for some $c>0$ and all $x\in\sB^c$,
and $\int_{\RR^n}\pi(x)\,\D x=1$.
Further, for $\beta\in[0,(1+\alpha(2-n))/2]$ and $x\in\RR^n$,  let
\begin{equation*}
	\sigma(x)\,\df\,\bigl(\pi(x)\bigr)^{-\beta}\,\Id_n\,,\qquad
	a(x)\,\df\, \sigma(x)\sigma(x)'=\bigl(\pi(x)\bigr)^{-2\beta}\,\Id_n\,,
\end{equation*}
and
\begin{equation*}
b(x)\,\df\,\frac{1}{2}\,\bigl(a(x)\,\nabla\log(\pi(x))+\nabla\cdot a(x)'\bigr)
\,=\,\frac{1-2\beta}{2}\,\bigl(\pi(x)\bigr)^{-2\beta}\nabla\log(\pi(x))\,.
\end{equation*}
Then, in \cite[Proposition 15]{Fort-Roberts-2005} it has been shown that the SDE
\begin{equation*}
	X(t)\,=\,x+\int_0^tb\bigl(X(s)\bigr)\,\D s+\int_0^t\sigma\bigl(X(s)\bigr)\, \D B_s\,,
	\qquad t\ge0\,,\quad x\in\RR^n\,,
\end{equation*} 
admits a weak solution
$(\Omega,\mathcal{F},\{\mathcal{F}_t\}_{t\ge0},\process{B},\process{X},\Prob)$,
which is a conservative strong Markov process with continuous sample paths.
Moreover, it is irreducible, aperiodic, every compact set is petite, and
$\uppi(\D x)\df\pi(x)\D x$ is its unique invariant probability measure.
Here, $(\Omega,\mathcal{F}, \{\mathcal{F}_t\}_{t\ge0},\process{B},\Prob)$ is a
standard $n$-dimensional Brownian motion.
Note also that according to It\^{o}'s formula $\process{X}$ satisfies
\hyperlink{MP}{\ttup{MP}} with
\begin{equation*}\mathcal{L}f(x)\,=\, \bigl\langle b(x),\nabla f(x)\bigr\rangle
+\frac{1}{2}\trace\bigl(a(x)\nabla^2f(x)\bigr)\,,\qquad x\in\RR^n\,.\end{equation*}

\begin{proposition}\label{P3.1}
	\begin{enumerate}
		\item[\ttup i]
		If $\beta\in\bigl[\alpha,\frac{1}{2}(1+\alpha(1-n))\bigr)$,
		then the assertions of \cref{T1.1}\,\ttup{iii}
		hold with $\Lyap(x)= 1+\Lyap_{\Id_n,\frac{\gamma}{\alpha}}(x)$ and
		$\eta= \frac{\gamma}{\alpha}$ for any $\gamma\in[\alpha,1+\alpha(2-n)-2\beta)$.
		
		\item[\ttup{ii}]
		If $\alpha\in(0,1/(n+1))$ and $\beta\in[0,\alpha)$, then the assertions of
		\Cref{T1.1}\,\ttup{i} and \ttup{ii} hold with
		\begin{equation*}
			\Lyap(x)= 1+\Lyap_{\Id_n,\frac{\gamma}{\alpha}}(x)\,,\quad
			\phi(t)= t^{1-\frac{2(\alpha-\beta)}{\gamma}}\,,\quad\text{and\ \ }
			\eta= \alpha^{-1}\bigl(\gamma-2(\alpha-\beta)\bigr)
		\end{equation*}
		for any
		$\gamma\in[3\alpha-2\beta,1+\alpha(2-n)-2\beta)$.
		
		\item[\ttup{iii}]
		Under the assumptions of \ttup{ii},
		$\uppi\in\cP_{\alpha^{-1}(1-\alpha n)-\iota}(\RR^n)$ for
		$\iota\in\bigl(0,\alpha^{-1}(1-\alpha(n+1))\bigr)$.
		Let $\rho\in(0,(1-(n+1)\alpha)\wedge2(\alpha-\beta))$
		and $\varepsilon\in\bigl[\alpha^{-1}\rho,2\alpha^{-1}(\alpha-\beta)\bigr)$
		be fixed.
		Then, for every $p\in\bigl[1,\alpha^{-1}(1-n\alpha-\rho)\bigr]$ and
		$\iota\in(0,2\alpha^{-1}(\alpha-\beta)-\varepsilon)$ there exist
		a positive constant $\Bar c$ and a
		diverging increasing sequence $\{t_n\}_{n\NN}\subset[0,\infty)$,
		depending on the above parameters, such that \cref{ET1.2A} in \cref{T1.2}
		holds with $\Lyap(x)$ as above, $\theta= \alpha^{-1}(1+\alpha(2-n)-2\beta-\rho)$,
		and $\vartheta= \alpha^{-1}(1-n\alpha-\rho)$.
	\end{enumerate}
\end{proposition} 

\begin{proof}
	\begin{itemize}
		\item [\ttup i]
		In \cite[Theorem~16\,\ttup i]{Fort-Roberts-2005} it has been shown that for
		$\beta\in\bigl[\alpha,\frac{1}{2}(1+\alpha(1-n))\bigr)$
		and $\gamma\in(0,1+\alpha(2-n)-2\beta)$
		the Foster-Lyapunov condition in \cref{ET1.1A} holds with $\Lyap(x)$ as above,
		$\phi(t)=t$ and $C=\Bar{\sB}_r$ for some $r>0$ large enough. Also, the relation in \cref{ET1.1B} easily follows from the form of $\Lyap(x)$ and
		$\phi(t)$, and the choice of $\eta$.

		\item [\ttup {ii}]
		In \cite[Theorem~16\,\ttup{ii}]{Fort-Roberts-2005} it has been shown that for
		$\alpha\in(0,1/n)$, $\beta\in[0,\alpha)$ and
		$\gamma\in(2(\alpha-\beta),1+\alpha(2-n)-2\beta)$, the Foster-Lyapunov condition
		in \cref{ET1.1A} holds with $\Lyap(x)$ and $\phi(t)$ as above and $C=\Bar{\sB}_r$
		for some $r>0$ large enough. The relation in \cref{ET1.1B} can again be easily verified due to the form of $\Lyap(x)$ and
		$\phi(t)$, and the choice of $\eta$.
		
		\item[\ttup{iii}]
		Since $\vartheta+\varepsilon-\nicefrac{1}{\alpha}+n-1\ge-1$, we have
$\int_{ \RR^n}\abs{x}^{\vartheta+\varepsilon}\,\uppi(\D x)\,=\,\infty$.
		The assertion now follows from \cref{T1.2} by taking $L(x)=\abs{x}$.
	\end{itemize}
	This completes the proof.
\end{proof}

\begin{remark}
	Observe that the rates obtained in \cref{P3.1}\,\ttup {ii} and \ttup{iii} match.
	Also, in \cref{P3.1}\,\ttup {ii} we assume that $\alpha\in\bigl(0,(n+1)^{-1}\bigr)$.
	Namely, for $\alpha\in\bigl[(n+1)^{-1},n^{-1}\bigr)$ it holds that
	$\int_{\RR^n}\abs{x}\,\uppi(\D x)=\infty$, and hence convergence in the
	$\sW_{p}$-distance cannot hold.
	On the other hand, in this case, \cite[Theorem~16\,\ttup{ii}]{Fort-Roberts-2005}
	shows subexponential convergence in the $f$-norm.
	In the following subsections we give examples of Markov processes which are
	ergodic in the
	$\sW_p$-distance but not in the $f$-norm.
	For additional results
	on ergodic properties of Langevin tempered diffusion processes with respect
	to the $f$-norm see \cite{Douc-Fort-Guilin-2009} and \cite{Fort-Roberts-2005}.
\end{remark}

\subsection{Ornstein-Uhlenbeck processes with jumps}

We next consider a class of It\^{o} processes with linear drift.
Let $H$ be an $n\times n$ matrix, and let $\process{L}$ be an $n$-dimensional L\'evy
process determined by L\'evy triplet $\bigl(b_L,a_L,\upnu_L(\D y)\bigr)$.
It is well known that the SDE
\begin{equation*}X(t)\,=\,x+H\int_0^tX(s)\,\D s+ L(t)\,,\qquad t\ge0\,,\quad x\in\RR^n\,,\end{equation*}
admits a unique conservative strong solution $\process{X}$ which is a strong Markov
process with c\`{a}dl\`{a}g sample paths (see e.g. \cite[Theorem 3.1 and Proposition 4.2]{Albeverio-Brzezniak-Wu-2010}).
In particular, $\process{X}$ is an It\^{o} process satisfying \hyperlink{MP}{\ttup{MP}}
with $b(x)=b_L+Mx$, $a(x)=a_L$, and $\upnu(x,\D y)=\upnu_L(\D y)$.
This process is known as an Ornstein-Uhlenbeck process with jumps.
In the case when $\process{L}$ is a standard Brownian motion,
$\process{X}$ is the classical Ornstein-Uhlenbeck process. 
If $H$ is a Hurwitz matrix (a square matrix whose eigenvalues have all strictly
negative real parts), it has been shown in \cite[Theorems 4.1 and 4.2]{Sato-Yamazato-1984}
that $\process{X}$ admits a unique invariant $\uppi\in\cP(\RR^n)$ if, and only if,
\begin{equation*}\int_{\sB^c}\log(|y|)\,\upnu_L(\D y)\,<\,\infty\,.\end{equation*}
Moreover, if this is the case, then
$\lim_{t\to\infty}\updelta_xP_t\bigl(f\bigr)\,=\,\uppi\bigl(f\bigr)$
for all $x\in\RR^n$ and $f\in C_b(\RR^n)$,
i.e. for any $x\in\RR^n$, the transition kernel $p(t,x,\D y)$ converges weakly,
as $t\to\infty$, to $\uppi(\D y)$.
However, this is not enough for $\sW_{p}$-convergence of $p(t,x,\D y)$ to
$\uppi(\D y)$ (see \cite[Theorem~6.9]{Villani-Book-2009}).
Assume additionally that $1\in\Theta_\upnu$, and let
$p\in[1,\theta_\upnu]\cap\Theta_\upnu$.
Since $H$ is Hurwitz, there exists $Q\in\cM_+$ such that $-(QH+H'Q)\in\cM_+$
(see \cite[Lemma~2.2]{Hurwitz}).
The left-hand side of \cref{ET1.4A} then reads
\begin{equation*}2\,\bigr\langle \varDelta_z\Tilde b(x),Qz\bigl\rangle \,=\,
2\,\bigl\langle Hz,Qz\bigr\rangle\,=\,
\bigl\langle (QH+H'Q)z,z\bigr\rangle\qquad \forall\,x,z\in\RR^n\,.\end{equation*}
Now, by setting
\begin{equation*}c(p)\df \inf_{z\in\RR^n}\frac{-p\,\bigl\langle (QH+H'Q)z,z\bigr\rangle}{2\,|z|^2_Q}\,,\end{equation*}
the assertions of \cref{T1.4} follow.
We remark here that this result does not necessarily imply ergodicity of $\process{X}$
in the $f$-norm.
Indeed, let $n=1$, and take $L_t\equiv0$ for $t\ge0$. 
Then it is easy to see that $X_t=x\,\E^{Ht}$ for $t\ge0$.
Thus, $\uppi(\D x)=\updelta_0(\D x)$, and $\updelta_x P_t$ converges to $\uppi(\D x)$,
as $t\to\infty$, in $\sW_p$-distance for any $p\ge1$, but clearly this convergence
cannot hold in the $f$-norm.

If $\theta_\upnu>1$, and $\process{X}$ satisfies the assumptions in
\cite[Theorem~3.1]{Arapostathis-Pang-Sandric-2019}
(which ensure that $\process{X}$ is irreducible and aperiodic, and that the support of
the corresponding irreducibility measure has nonempty interior),
then according to \cite[Proposition 4.3]{Albeverio-Brzezniak-Wu-2010} and
\cite[Theorems 5.1 and 7.1]{Tweedie-1994}
(which imply that every compact set is petite for $\process{X}$) 
the conclusions of \cref{T1.3}\,\ttup{ii} hold true for any
$\theta\in(1,\theta_\upnu]\cap\Theta_\upnu$.
If $\theta_\upnu>0$,
then under the same assumptions as above,
\cite[Theorem~5.2]{Down-Meyn-Tweedie-1995},
\cite[Proposition 6.1]{Meyn-Tweedie-AdvAP-II-1993}, and
\cite[Theorem~4.2]{Meyn-Tweedie-AdvAP-III-1993}
(and \cite[Proposition 4.3]{Albeverio-Brzezniak-Wu-2010},
and \cite[Theorems 5.1 and 7.1]{Tweedie-1994}) imply that for any
$\theta\in(0,\theta_\upnu]\cap\Theta_\upnu$ the process $\process{X}$ is exponentially
ergodic in the $f$-norm with $f(x)=\Lyap_{Q,\theta}(x)+1$.
However, this does not necessarily imply ergodicity of $\process{X}$ in the
$\sW_{p}$-distance.
To see this take again $n=1$, and let $\process{L}$ be a one-dimensional symmetric
$\alpha$-stable L\'evy process with $\alpha\in(0,1)$ and symbol
(characteristic exponent) $q(\xi)=\abs{\xi}^\alpha$.
Thus, $1\notin\Theta_\upnu$, and $\theta_\upnu=\alpha$.
We claim that $\uppi\notin\cP_1(\RR)$. Assume this is not the case. Then,
\begin{equation*}
\int_{ \RR}\int_{ \RR}\abs{y}\,p(t,x,\D y)\,\uppi(\D x)\,=\,
\int_{ \RR}\abs{x}\,\uppi(dx)\,<\,\infty\qquad \forall\,t\ge0\,.
\end{equation*}
In particular, for every $t>0$ it holds that $\int_{ \RR} \abs{y}\,p(t,x,\D y)<\infty$,
$\uppi$-a.e.
On the other hand, according to \cite[Theorem~3.1]{Sato-Yamazato-1984}, we have
\begin{equation*}P_tf(x)\,=\,\int_{\RR}f(\E^{Ht}x+y)\,\upmu_t(\D y)\end{equation*} for all $t\ge0$, 
$x\in\RR$ and $f\in\cB_b(\RR),$
where $\upmu_t(\D y)$ is a probability measure on $\RR$ with characteristic function
\begin{equation*}\Hat{\upmu}_t(\xi)\,=\,\E^{-\int_0^tq(\E^{Hs}\xi)\,\D s}
\,=\,\E^{\frac{1-\E^{\alpha Ht}}{\alpha H}\,\abs{\xi}^\alpha}\,,
\qquad t\ge0\,,\quad \xi\in\RR\,,\end{equation*}
and $\cB_b(\RR)$ denotes the space of bounded functions in $\cB(\RR)$.
Hence, $\upmu_t(\D y)$ is the law of a symmetric $\alpha$-stable random variable.
Now, the monotone convergence theorem implies that
\begin{equation*}\int_{\RR}\abs{y}\,p(t,x,\D y)\,=\,
\int_{ \RR}\abs{\E^{Ht}x+y}\,\upmu_t(\D y)\,<\,\infty
\qquad \forall\,(t,x)\in[0,\infty)\times\RR\,,\end{equation*}
which is impossible.

Let us mention that ergodic properties of Ornstein-Uhlenbeck processes with jumps
in the $f$-norm, and in particular in the total variation norm,
have been considered in \cite{Kevei16,Masuda-04,Schiling-Wang-2012,Wang11,Wang12}.

\subsection{Piecewise Ornstein-Uhlenbeck processes with jumps}\label{S3.3}

We extend the results from the previous subsection to a class of It\^{o} processes
with a piecewise linear drift.
Consider an $n$-dimensional SDE of the form
\begin{equation}\label{PL-SDE}
X(t)\,=\,x+\int_0^t\Bar b\bigl(X(s)\bigr)\,\D{s} +\int_0^t\sigma\bigl(X(s)\bigr)\,
\D B(s)+ L(t)\,,\qquad t\ge0\,,\quad x\in\RR^n\,,
\end{equation} 
where 
\begin{enumerate}
	\item [\ttup{i}]
	the function $\Bar b\colon\RR^n\to\RR^n$ is given by
	\begin{equation*}
		\Bar b(x) \,=\, l-M(x-\langle e,x\rangle^+v)-\langle e,x\rangle^+\varGamma v\,,
	\end{equation*}
	where $l \in \RR^n$, $v\in\RR^n$ has nonnegative components and satisfies
	$\langle e,v\rangle=1$ with $e = (1,\dotsc,1)'\in\RR^n$,
	$M\in\RR^{n\times n}$ is a nonsingular M-matrix
	such that the vector $e'M$ has nonnegative components, 
	and $\varGamma=\diag(\gamma_1,\dotsc,\gamma_n)$ with $\gamma_i\ge0$ for
	$i=1,\dotsc,n$\,;
	
	\item [\ttup{ii}] 
	$\process{B}$ is a standard $m$-dimensional Brownian motion, and
	the covariance function $\sigma\colon \RR^n \to \RR^{n\times m}$ is locally
	Lipschitz continuous and satisfies, for some $c>0$,
	\begin{equation*}
		\norm{\sigma(x)}^2\,\le\,c\,(1+\abs{x}^2)\qquad \forall\,x\in\RR^n\,;
	\end{equation*}
	
	\item [\ttup{iii}]
	$\process{L}$ is a $n$-dimensional
	pure-jump L\'evy process specified by a drift $b_L\in\RR^n$ and L\'evy measure
	$\upnu_L(\D y)$. 
\end{enumerate}

Recall that a $n\times n$ matrix $M$ is called an M-matrix
if it can be expressed as $M=\mu\mathbb{I}_n-N$ for some $\mu>0$ and some
nonnegative $n\times n$ matrix $N$ with the property that $\rho(N)\le \mu$, where
$\mathbb{I}_n$ and $\rho(N)$ denote the $n\times n$ identity matrix
and the spectral radius of $N$, respectively. 
Clearly, the matrix $M$ is nonsingular if $\rho(N)<\mu$.
It is well known that the SDE in \cref{PL-SDE} admits a unique conservative strong
solution $\process{X}$ which is a strong Markov process with c\`{a}dl\`{a}g sample paths (see e.g. \cite[Theorem 3.1 and Proposition 4.2]{Albeverio-Brzezniak-Wu-2010}). 
In particular, $\process{X}$ is an It\^{o} process satisfying \hyperlink{MP}{\ttup{MP}}
with $b(x)=b_L+\Bar b(x)$, $a(x)=\sigma(x)\sigma(x)'$, and $\upnu(x,\D y)=\upnu_L(\D y)$. 
This process is often called a piecewise Ornstein-Uhlenbeck process with jumps. 
It arises as a limit of the suitably scaled queueing processes of multiclass many-server
queueing networks with heavy-tailed (bursty) arrivals and/or asymptotically negligible
service interruptions.
In these models, if the scheduling policy is based on a static priority assignment
on the queues, then the vector $v$ in the limiting diffusion \cref{PL-SDE} corresponds
to a constant control.
The process $\process{X}$ also arises in many-server queues with phase-type service times,
where the constant vector $v$ corresponds to the probability distribution of the phases.
For a multiclass queueing network with independent heavy-tailed arrivals, the process
$\process{L}$ is an anisotropic L\'evy process consisting of independent one-dimensional
symmetric $\alpha$-stable components.
Under service interruptions, $\process{L}$ is either a compound Poisson process,
or an anisotropic L\'evy process described above together with a compound Poisson
component.
More details on these queueing models can be found in
\cite[Section~4]{Arapostathis-Pang-Sandric-2019}.

We first discuss the case when $\varGamma v=0$.
This corresponds to the case when the control gives lowest priority to queues
whose abandonment rate is zero.
When $1\in\Theta_\upnu$, we define
\begin{equation}\label{L}
\Tilde l\,\df\, l+b_L+\int_{\sB^c}y\,\upnu_L(\D y)\,.
\end{equation}

\begin{proposition}\label{P3.2}
	In addition to the assumptions of \cite[Theorem~3.1]{Arapostathis-Pang-Sandric-2019}
	\textup{(}which ensure that $\process{X}$ is irreducible and aperiodic with
	irreducibility measure having support with nonempty interior\textup{)},
	suppose that $\varGamma v=0$, $2\in\Theta_\upnu$, and
	$\bigl\langle e,M^{-1}\Tilde{l}\bigr\rangle<0$.
	\begin{enumerate}
		\item [\ttup{i}]
		If \begin{equation}\label{A}
		\limsup_{\abs{x}\to\infty}\,\frac{\norm{a(x)}}{\abs{x}}\,=\,0\,,
		\end{equation}
		then there exists $Q\in\cM_+$ such that the assertions of \cref{T1.3}\,\ttup{i}
		hold true with $\vartheta=1$.
		
		\item [\ttup{ii}]
		If $a(x)$ is bounded, and
		$\int_{\sB^c}\E^{\theta\rvert y\rvert}\,\upnu_L(\D y) < \infty$
		for some $\theta>0$, then there exists $Q\in\cM_+$ such that the assertions
		of \cref{T1.3}\,\ttup{iii} hold.
	\end{enumerate}
\end{proposition}

\begin{proof}
	\begin{itemize}
		\item [\ttup{i}] In \cite[Theorem~3.2\,\ttup{i}]{Arapostathis-Pang-Sandric-2019} 
		it has been shown that there exist $Q\in\cM_+$, $\Bar c=\Bar c(\theta)>0$,
		and $\Tilde c=\Tilde c(\theta)>0$, such that for any
		$\theta\in[1,\theta_\upnu]\cap\Theta_\upnu$, we have
		\begin{equation*}\mathcal{L} \Lyap_{Q,\theta}(x)\,\le\, \Bar c
		-\Tilde c\, \Lyap_{Q,\theta-1}(x)\qquad \forall\,x\in\RR^n\,.\end{equation*}
		It is easy to see that the above relation implies that there exist $r>0$, $\Hat c>0$,
		and $\Breve{c}>0$, such that
		\begin{equation*}\mathcal{L} \Lyap(x)\,\le\, \Hat c\,\Ind_{\Bar{\sB}_r}(x)
		-\Breve c\, \bigl(\Lyap(x)\bigr)^{\nicefrac{(\theta-1)}{\theta}}
		\qquad \forall\,x\in\RR^n\,,\end{equation*}
		with $\Lyap(x)=\Lyap_{Q,\theta}(x)+1$.
		The assertion now follows from \cref{T1.3}\,\ttup{i}, together with
		\cite[Proposition 4.3]{Albeverio-Brzezniak-Wu-2010},
		\cite[Theorem~3.4]{Douc-Fort-Guilin-2009},
		and \cite[Theorems 5.1 and 7.1]{Tweedie-1994}. 
		
		\item [\ttup{ii}]
		Let $\Tilde{b}(x)\,\df\, \Bar{b}(x) + \Tilde{l} -l$.
		As shown in the proof of \cite[Theorem~3.2\,\ttup{ii}]{Arapostathis-Pang-Sandric-2019}, there exist $Q\in\cM_+$, $\Bar c=\Bar c(\zeta)>0$
		and $\Tilde c=\Tilde c(\zeta)>0$, such that
		for any $\zeta\in\bigl(0,\theta\norm{Q}^{-\nicefrac{1}{2}}\bigr)$,
		\begin{equation*}
			\bigl\langle \Tilde{b}(x), \nabla \widetilde{V}_{Q,\zeta}(x)\bigr\rangle
			\,\le\, \bar c- \tilde c \, \widetilde{V}_{Q,\zeta}(x)
			\qquad \forall\,x\in\RR^n\,. 
		\end{equation*}
		This together with \cref{L2.2}\,(iii) imply
		that, for any  $\zeta>0$ sufficiently small, there exist
		$\hat c=\hat c(\zeta)>0$
		and $\check c=\check c(\zeta)>0$, such that
		\begin{equation*}\mathcal{L} \widetilde\Lyap_{Q,\zeta}(x)\,\le\, \hat c-\check c \,\widetilde\Lyap_{Q,\zeta}(x)
		\qquad \forall\,x\in\RR^n\,.\end{equation*} 
		Again, It is straightforward to see that the above relation implies that there
		exist $r>0$, $\breve c>0$ and $\mathring{c}>0$, such that
		\begin{equation*}\mathcal{L} \widetilde\Lyap_{Q,\zeta}(x)\,\le\, \breve c\,\Ind_{\Bar{\sB}_r}(x)
		-\mathring c\, \widetilde\Lyap_{Q,\zeta}(x)\qquad\forall\, x\in\RR^n\,.\end{equation*} 
		The assertion now follows from \cref{T1.3}\,\ttup{iii}, and
		the results from \cite{Albeverio-Brzezniak-Wu-2010,Douc-Fort-Guilin-2009,Tweedie-1994}
		cited in part \ttup{i}.\endproof
	\end{itemize}
\end{proof}

\begin{remark}
	It has been shown
	in \cite[Theorem~3.3\,\ttup{b} and Lemma~5.7]{Arapostathis-Pang-Sandric-2019}
	that the assumptions $1\in\Theta_\upnu$ and
	$\langle e,M^{-1}\Tilde{l}\rangle<0$ are both necessary
	for the existence of an invariant probability measure of $\process{X}$.
	Using this, we can exhibit an example where we have ergodicity with respect to the
	$f$-norm but not with respect to $\sW_p$-distance.
	Suppose that $\varGamma v=0$,
	$\langle e,M^{-1}\Tilde{l}\rangle<0$, $a(x)$ satisfies \cref{A}, and
	$\process{L}$ is a rotationally invariant $\alpha$-stable process with $\alpha \in (1,2)$.
	Then \cite[Theorem~3.1\,\ttup{i}]{Arapostathis-Pang-Sandric-2019} shows that
	$\process{X}$ admits a unique invariant $\uppi\in\cP_{\alpha-1-\iota}(\RR^n)$
	for $\iota\in(0,\alpha-1)$, and 
	\begin{equation*}
		\lim_{t\to\infty}t^{\alpha-1-\iota}\,\norm{\updelta_x P_t(\cdot)
			-\uppi(\cdot)}_{\mathrm{TV}}\,=\,0\end{equation*} for all $x\in\RR^n$ and $\iota\in (0,\alpha-1)$. Here, $\norm{\cdot}_\mathrm{TV}$ stands for the total variation norm, i.e.
	the $f$-norm with $f(x)\equiv1$.
	However, $\int_{\RR^n} \abs{x}\, \uppi(\D x)=\infty$
	by \cite[Theorem~3.4\,(b)]{Arapostathis-Pang-Sandric-2019}, so we cannot
	have convergence in $\sW_1$-distance.
\end{remark}

We next exhibit a lower bound on the polynomial rate of convergence
in \cref{P3.2}\,\ttup{i}, which is analogous to
\cite[Theorem~3.4]{Arapostathis-Pang-Sandric-2019}.
We let 
\begin{equation}\label{THETA}
\Tilde\theta_\upnu \,\df\, 
\sup\biggl\{ \theta\ge 0\colon \int_{\RR^n}\bigl(\langle e,M^{-1}y\rangle^+\bigr)^\theta\,
\upnu_L(\D y)<\infty\biggr\}\,.
\end{equation}
Note that, in general, $\Tilde\theta_\upnu\ge\theta_\upnu$.
In \cite{Arapostathis-Pang-Sandric-2019} it is assumed that
$\process{L}$ is a compound Poisson process with drift $b_L$,
and L\'evy measure $\upnu_L(\D y)$ which is supported on a half-line of the form
$\{\zeta w\colon \zeta\in[0,\infty)\}$ with $\langle e,M^{-1}w\rangle>0$,
and $a(x)$ satisfies \cref{A}.
This implies that $\Tilde\theta_\upnu=\theta_\upnu$, and subsequently,
this equality is used in the
proof of \cite[Lemma~5.7\,\ttup{b}]{Arapostathis-Pang-Sandric-2019}
to establish that, provided $\varGamma v=0$, 
$\int_{\RR^n}\bigl(\langle e,M^{-1}x\rangle^+\bigr)^{p-1}\,
\uppi(\D x)<\infty$ implies $p\in\Theta_\upnu$ for $p>1$. 
We use this fact, namely
that the conclusions of \cite[Lemma~5.7\,\ttup{b}]{Arapostathis-Pang-Sandric-2019}
hold under the weaker assumption that $\Tilde\theta_\upnu=\theta_\upnu$
in the proof of the following proposition.

\begin{proposition}\label{P3.3}
	In addition to the assumptions of \cite[Theorem~3.1]{Arapostathis-Pang-Sandric-2019},
	assume that $\varGamma v=0$, $\bigl\langle e,M^{-1}\Tilde{l}\bigr\rangle<0$,
	and $\Tilde\theta_\upnu=\theta_\upnu\in(2,\infty)$.
	Then, due to \cref{P3.2}\,\ttup{i}, $\process{X}$ admits a unique invariant
	$\uppi\in\cP_{\theta_\upnu-1-\iota}(\RR^n)$,
	$\iota\in(0,\theta_\upnu-1)$.
	Next, fix $\rho\in(0,(\theta_\upnu-2)\wedge1)$ and $\varepsilon\in(\rho,1)$.
	Then, for any $p\in[1,\theta_\upnu-\rho-1]$ and $\iota\in(0,1-\varepsilon)$ there
	exist $\Bar c>0$ and a diverging increasing sequence
	$\{t_n\}_{n\in\NN}\subset[0,\infty)$, depending on these parameters, such that
	\cref{ET1.2A} holds with $\theta =\theta_\upnu-\rho$, $\vartheta=\theta-1$,
	and $\Lyap(x)=\Lyap_{Q,\theta}(x)+1$,
	where $Q\in\cM_+$ is given in \cref{P3.2}\,\ttup i. 
\end{proposition}

\begin{proof}
	Observe first that $\vartheta+\varepsilon>\theta_\upnu-1$.
	Thus, according to \cite[Lemma~5.7\,\ttup{b}]{Arapostathis-Pang-Sandric-2019}, we
	have
	\begin{equation*}\int_{\RR^n}\bigl(\langle e,M^{-1}x\rangle^+\bigr)^{\vartheta+\varepsilon}\,
	\uppi(\D x)\,=\,\infty\,.\end{equation*}
	The assertion now follows from the proof of \cref{P3.2}\,\ttup i
	(together with \cite[Proposition 4.3]{Albeverio-Brzezniak-Wu-2010},
	\cite[Theorem~3.4]{Douc-Fort-Guilin-2009}, and
	\cite[Theorems 5.1 and 7.1]{Tweedie-1994}),
	and \cref{T1.2} by setting $L(x)=\langle e,M^{-1}x\rangle^+$ and
	$\phi(t)= t^{\nicefrac{(\theta-1)}{\theta}}$.
\end{proof}

We now discuss the case when $\varGamma v\neq0$.
For $x\in\RR^n$, we write $x\ge0$ ($x\gneqq 0$) to indicate that all components of $x$
are nonnegative (nonnegative and at least one is strictly positive).
Also, for $x,y\in\RR^n$ we write $x\ge y$ if, and only if, $x-y\ge0$. 

\begin{proposition}\label{P3.4}
	In addition to the assumptions of \cite[Theorem~3.1]{Arapostathis-Pang-Sandric-2019},
	suppose that $\theta_\upnu>0$, 
	\begin{equation}\label{EP3.4A}
	\limsup_{|x|\to\infty}\,\frac{\lVert a(x)\rVert}{|x|^2} \,=\, 0\,,
	\end{equation}
	and that one of the following holds:
	\begin{enumerate}
		\item[\ttup i]
		$Mv\ge\varGamma v\gneqq0$; 
		\item[\ttup {ii}]
		$M=\diag(m_1,\dotsc,m_d)$ with $m_i>0$,
		$i=1,\dotsc,n$, and $\varGamma v\ne 0$.
	\end{enumerate}
	Then there exists $Q\in\cM_+$ such that the assertions of
	\Cref{T1.3}\,\ttup{ii} hold true.
\end{proposition}
\begin{proof}
	In \cite[Theorem~3.5]{Arapostathis-Pang-Sandric-2019} 
	it has been shown that there exist $Q\in\cM_+$, $\Bar c=\Bar c(\theta)>0$,
	and $\Tilde c=\Tilde c(\theta)>0$, such that for any
	$\theta\in(0,\theta_\upnu]\cap\Theta_\upnu$, we have
	\begin{equation*}\mathcal{L} \Lyap_{Q,\theta}(x)\,\le\, \Bar c-\Tilde c \,\Lyap_{Q,\theta}(x)
	\qquad\forall\, x\in\RR^n\,.\end{equation*}
	As in \cref{P3.2}, it is easy to see that the above relation implies that there
	exist  $r>0$, $\Hat c>0$ and $\Breve{c}>0$, such that
	\begin{equation*}\mathcal{L} \Lyap(x)\,\le\, \Hat c\,\Ind_{\Bar{\sB}_r}(x)-\Breve c\, \Lyap(x)
	\qquad \forall\,x\in\RR^n\,,\end{equation*}
	with $\Lyap(x)=\Lyap_{Q,\theta}(x)+1$.
	The assertion now follows from \cref{T1.3}\,\ttup{ii}, together with
	the results from \cite{Albeverio-Brzezniak-Wu-2010,Douc-Fort-Guilin-2009,Tweedie-1994}
	cited in the proof of \cref{P3.3}. 
\end{proof}

In the case when $\varGamma v\neq0$ (under \ttup i or \ttup{ii} in \cref{P3.4})
the dynamics are contractive
in the $\sW_p$-distance. This is shown by establishing
an asymptotic flatness (uniform dissipativity) property for $\process{X}$. 
As a consequence, we assert exponential ergodicity of $\process{X}$ with respect to
$\sW_p$, without assuming irreducibility and aperiodicity, i.e.
we allow the SDE in \cref{PL-SDE} to be degenerate.

\begin{proposition}\label{P3.5}
	Suppose that $2\in\Theta_\upnu$, $\sigma(x)$ is Lipschitz continuous,
	and either \ttup i or \ttup{ii} in \cref{P3.4} holds.
	Then there exists $Q\in \mathcal{M}_+$ such that the matrices
	\begin{equation}\label{EP3.5A}
	MQ+QM\,,\quad\text{and}\quad \bigl(M- ev'(M-\varGamma)\bigr)Q 
	+ Q\bigl(M- (M-\varGamma)v e'\bigr)
	\end{equation}
	are in $\mathcal{M}_+$.
	Let
	$\underline\kappa$ denote the smallest eigenvalue
	of the positive definite matrices in \cref{EP3.5A}, and $\overline\lambda_Q$,
	$\underline\lambda_Q$ denote the largest, smallest eigenvalue of $Q$, respectively.
	For $p\ge1$, let
	\begin{equation*}
		c(p) \,\df\, \frac{p}{2}
		\left(\frac{\underline\kappa}
		{\overline\lambda_Q}
		-\frac{(p-1)\Lip^2\bigl(\sqrt{Q}\,\sigma\bigr)}{\underline{\lambda}_Q}\right)\,,
	\end{equation*}
	where $\Lip(\sqrt{Q}\,\sigma)$ is the Lipschitz constant of $\sqrt{Q}\,\sigma(x)$ 
	with respect to the Hilbert-Schmidt norm, and
	suppose that $c(p)>0$ for some $p\in[2,\theta_\upnu]\cap\Theta_\upnu$. 
	Then the assertions of \cref{T1.4} hold true.
	If $\sigma(x)\equiv\sigma$ and $1\in\Theta_\upnu$, the assertions of \cref{T1.4}
	hold true 
	for any $p\in[1,\theta_\upnu]\cap\Theta_\upnu$.
\end{proposition}

\begin{proof}
	Existence of the matrix $Q$ has been proven in
	\cite[Theorem~3.5]{Arapostathis-Pang-Sandric-2019}.
	We prove that \cref{ET1.4A} holds with $c(p)$ defined above. First, clearly,
	\begin{equation}
	\begin{aligned}\label{PP3.5A}
	\trace\,\bigl(\Tilde{a}(x;y-x) Q\bigr)+(p-2)\,
	\bnorm{\sqrt{Q}\,\varDelta_{y-x}\sigma(x)}^{2}\,&\le\,
	(p-1)\Lip^2\bigl(\sqrt{Q}\,\sigma\bigr)\,\abs{y-x}^2\,\\
	&\le\, \frac{(p-1)\Lip^2\bigl(\sqrt{Q}\,\sigma\bigr)}
	{\underline\lambda_Q}\,\abs{y-x}_Q^2
	\end{aligned}
	\end{equation}
	for all $x,y\in\RR^n$.
	We next discuss the term
	$\bigl\langle \varDelta_{y-x}\Tilde b(x),Q(y-x)\bigr\rangle$ for $x,y\in\RR^n$.
	Clearly, $\varDelta_{y-x}\Tilde b(x)=\varDelta_{y-x} \Bar b(x)$ for $x,y\in\RR^n$.
	With $\Hat{v}=-M^{-1}(M v - \varGamma v)$,
	we have
	$\Bar b(x) = l -M (x+\langle e,x\rangle^+\,\Hat{v})$.
	If both $x$ and $y$ are on the same half-space, i.e.
	$\langle e,x\rangle\ge0$ and $\langle e,y\rangle\ge0$,
	or the opposite, then
	\begin{equation*}\bigl\langle \varDelta_{y-x}\Tilde b(x),Q(y-x)\bigr\rangle\,\le\,
	- \frac{\underline\kappa}{2}\, \abs{y-x}^2\,.\end{equation*}
	So suppose, without loss of generality, that
	$\langle e,x\rangle\ge0$ and $\langle e,y\rangle\le0$.
	Then we have
	\begin{subequations}
		\begin{align}
			\bigl\langle y-x, Q \,\Bar b(x) \bigr\rangle &\,=\,
			\bigl\langle y-x,Q\,l\bigr\rangle -\bigl\langle y-x, Q M x\bigr\rangle 
			-\bigl\langle y-x, Q M\Hat{v}e' x\bigr\rangle\label{PP3.5B}\\
			\bigl\langle y-x, Q\, \Bar b(y) \bigr\rangle &\,=\,
			\bigl\langle y-x,Q\,l\bigr\rangle -\bigl\langle y-x, Q M y\bigr\rangle\,.
			\label{PP3.5C}
		\end{align}
	\end{subequations}
	We distinguish two cases.
	\begin{enumerate}
		\item[\ttup i]
		$\bigl\langle y-x, Q M\Hat{v}e' x\bigr\rangle\le0$.
		Then of course subtracting \cref{PP3.5B} from \cref{PP3.5C}, we obtain
		\begin{align*}
			\bigl\langle \varDelta_{y-x}\Tilde b(x),Q(y-x)\bigr\rangle &\,=\,
			- \bigl\langle y-x, QM (y-x)\bigr\rangle
			+ \bigl\langle y-x, Q M\Hat{v}e' x\bigr\rangle\\
			&\,\le\,- \bigl\langle y-x, QM (y-x)\bigr\rangle\\
			& \,\le\, -
			\frac{\underline\kappa}{2}\,\abs{y-x}^2\,.
		\end{align*}
		\item [\ttup{ii}]
		$\bigl\langle y-x, Q M\Hat{v}e' x\bigr\rangle>0$.
		Since $\langle e,x\rangle\ge0$, we must have
		$\bigl\langle y-x, Q M\Hat{v}\bigr\rangle>0$.
		This in turn implies, since $\langle e,y\rangle\le0$, that
		\begin{equation}\label{PP3.5D}
		\bigl\langle y-x, Q M\Hat{v}e' y\bigl\rangle\,\le\,0\,.
		\end{equation}
	\end{enumerate}
	Adding \cref{PP3.5B,PP3.5D} and subtracting
	\cref{PP3.5C} from the sum, we obtain
	\begin{equation}\label{PP3.5E}
	\begin{aligned}
	\bigl\langle \varDelta_{y-x}\Tilde b(x),Q(y-x)\bigr\rangle &\,\le\,
	- \bigl\langle y-x, QM (y-x)\bigr\rangle
	- \bigl\langle y-x, Q M\Hat{v}e' (y-x)\bigr\rangle\\
	&\,\le\,- \bigl\langle y-x, QM(\Id_n+\Hat{v}e') (y-x)\bigr\rangle \,\le\,
	- \frac{\underline\kappa}{2}\,\abs{y-x}^2\,.
	\end{aligned}
	\end{equation}
	Finally, combining \cref{PP3.5A,PP3.5E}, we obtain
	\begin{align*}
		2\,\bigl\langle \varDelta_{y-x}\Tilde b(x),Q(y-x)\bigr\rangle
		+\trace\,\bigl(\Tilde{a}(x;y-x) Q\bigr)&+(p-2)\,
		\bnorm{\sqrt{Q}\,\varDelta_{y-x}\sigma(x)}^{2}\\
		&\,\le\, \left(- \frac{\underline\kappa}{\overline\lambda_Q}+
		\frac{(p-1)\Lip^2\bigl(\sqrt{Q}\,\sigma\bigr)}{\underline\lambda_Q}\right)\abs{y-x}_Q^2\\
		&\,=\, -\frac{2\, c(p)}{p}\, \abs{y-x}_Q^2\qquad\forall\, x,y\in\RR^n\,,\end{align*}
	thus completing the proof.
\end{proof}

The hypothesis in \cref{P3.5} that $c(p)>0$
is, of course, always true if $\sigma(x)\equiv\sigma$, in which
case we have $c(p)=p\frac{\underline{\kappa}}{2\overline{\lambda}_Q}$.
This is the scenario for multiclass queueing models with service interruptions
described in \cite[Section~4.2]{Arapostathis-Pang-Sandric-2019}.

Some examples of degenerate SDEs of the form \cref{PL-SDE} for which \cref{P3.5}
is applicable are the following.
\begin{enumerate}
	\item[\ttup i]
	$\process{L}$ is given by $L(t) = R \Tilde{L}(t)$ for $t\ge0$, where
	$R \in \mathbb{R}^{n\times r}$ has rank smaller than $\min\{n,r\}$,
	and $\{\Tilde{L}(t)\}_{t\ge0}$ is a $r$-dimensional L\'evy process.
	As a special case $\{\Tilde{L}(t)\}_{t\ge0}$ may be composed of mutually
	independent $\alpha$-stable processes.
	This is the case in the queueing example described below.
	
	\item[\ttup{ii}]
	$\process{L}$ is a degenerate subordinate Brownian motion,
	as studied in \cite{zhang2014densities}.
\end{enumerate}

The following is an example of a degenerate SDE that arises in applications
for which \cref{P3.4} is applicable.
Consider a two class $GI/M/k+M$ queue with class-1 jobs having a Poisson process,
and class-2 jobs having a heavy-tailed renewal arrival process.
Service and patience times are exponentially distributed with rates $m_i$ and
$\gamma_i$ for $i=1,2$, respectively.
Assume that the arrival, service and abandonment processes are mutually independent,
and that the number of servers is $k$. 
Consider a sequence of such models indexed by $k$, operating in the critically loaded
asymptotic modified Halfin-Whitt regime as $k\to\infty$.
Let $\process{A^k_i}$ denote the arrival process for class $i=1,2$,
with arrival rates $\lambda^k_i$. 
Assume that $m_i$ and $\gamma_i$ for $i=1,2$ are independent of $k$, and that
$\frac{\lambda^k_i}{k}\to \lambda_i>0$ as $k\to\infty$, for $i=1,2$.
The arrival process $\process{A^k_1}$ satisfies a functional central
limit theorem (FCLT) with a Brownian motion limit
$\process{\Hat{A}_1}=\process{\sqrt{\lambda_1} B_1}$,
where $\process{B_1}$ is a standard Brownian motion, i.e.
\begin{equation*}\process{\Hat{A}^k_1}\, =\,
\bigl\{k^{-\nicefrac{1}{2}}(A^k_1(t) - \lambda^k_1 t)\bigr\}_{t\ge0}\,
\xRightarrow[k\to\infty]{\mathrm{J}_1}\, \process{\Hat{A}_1}\,.\end{equation*}
Here, $\xRightarrow[]{\mathrm{J}_1}$ denotes the convergence in the space
$D=D([0,\infty), \RR)$
of c{\`a}dl{\`a}g functions endowed with the Skorokhod $\mathrm{J}_1$ topology. 
We assume that the arrival process $\process{A^k_2}$ satisfies a FCLT with a symmetric
$\alpha$-stable L\'evy process $\process{\Hat{A}_2}$, $\alpha \in (1,2)$, in the limit,
i.e.
\begin{equation*}
	\process{\Hat{A}^k_2} \,=\,
	\bigl\{k^{-\nicefrac{1}{\alpha}}(A^k_2(t) - \lambda^k_2 t)\bigr\}_{t\ge0}
	\,\xRightarrow[k\to\infty]{\mathrm{M}_1}\, \process{\Hat{A}_2}\,.
\end{equation*}
Here, $\xRightarrow[]{\mathrm{M}_1}$
denotes the convergence in the space $D$ with the $\mathrm{M}_1$ topology. 
Let $\rho^k_i=\frac{\lambda^k_i}{km_i}$ and $\rho_i = \frac{\lambda_i}{m_i}$
for $i=1,2$. 
The modified Halfin-Whitt regime requires the parameters satisfy
\begin{equation*}
	\lim_{k\to\infty}k^{1-\nicefrac{1}{\alpha}} \left(1- \sum_{i=1}^2 \rho_i^k\right)
	\,=\, \Hat{\rho}\in \RR\,,\qquad\text{and}\qquad
	\sum_{i=1}^2 \rho_i \,=\, 1\,.
\end{equation*}
In addition, we assume that
$k^{-\nicefrac{1}{\alpha}} (\lambda^k_i - k \lambda_i)\to l_i$ as $k\to\infty$
for $i=1,2$.
Next, let $\process{X^k_i}$ denote the number of class-$i$ jobs in the system.
Define the scaled processes
$\Hat{X}^k_i(t)=k^{-\nicefrac{1}{\alpha}} (X^k_i(t)-k \rho_it)$ for $t\ge0$. 
Let $\process{U^k_i}$ be the scheduling control process, representing allocations of
service capacity to class $i$.
Let $\Hat{X}^k(t) = \bigl(\Hat{X}^k_1(t), \Hat{X}^k_2(t)\bigr)'$
and $U^k(t)= \bigl(U^k_1(t),U^k_2(t)\bigr)'$ for $t\ge0$. 
We consider work conserving and preemptive scheduling policies resulting in
constant controls at the limit, i.e.
$\process{U^k} \xRightarrow[k\to\infty]{\mathrm{J}_1} \process{V}$,
where $V(t)=v$ for $t\ge0$ with $v\in\RR^2$ being a probability vector.
Then, as in \cite[Theorem~4.1]{Arapostathis-Pang-Sandric-2019}, 
it can shown that
$\process{\Hat{X}^k} \xRightarrow[k\to\infty]{\mathrm{M}_1} \process{X}$,
where the limit process $\process{X}$
is a solution to the following two-dimensional degenerate $\alpha$-stable driven SDE: 
\begin{align*}
	\D X_1(t) &\,=\, \Bigl( l_1-m_1 (X_1(t)
	- \langle e, X(t)\rangle^+ v_1) - \gamma_1 \langle e, X(t)\rangle^+ v_1 \Bigr)
	\D t\,, \\
	\D X_2(t) &\,=\, \Bigl(l_2 -m_2 (X_1(t)
	- \langle e, X(t)\rangle^+ v_2) -\gamma_2
	\langle e, X(t)\rangle^+ v_2 \Bigr) \D t
	+ \D \Hat{A}_2(t)\,,
\end{align*}
which is \cref{PL-SDE} with $l=(l_1,l_2)'$,
$M=\diag(m_1,m_2)$,
$\varGamma=\diag(\gamma_1,\gamma_2)$, $\sigma(x)=(0,0)'$,
and $L(t)=(0,\Hat{A}_2(t))'$ for $t\ge0$.
Observe that the process $\process{X}$ does not fall into any of the four categories in
\cite[Theorem~3.1]{Arapostathis-Pang-Sandric-2019}.
In fact, one can consider multiple classes of jobs with all heavy-tailed arrival
processes that have different scaling parameters $\alpha_i$'s for $i=1,\dots,\Bar{k}$,
in their corresponding FCLTs.
The centered queueing process should be scaled as
$k^{-\nicefrac{1}{\alpha}}$, where $\alpha\df\min_{i=1,\dots,\Bar{k}}\{\alpha_i\}$, and 
the limit process has the components $\process{X_i}$ driven by independent
$\alpha$-stable processes if the arrival process of class $i$ has the parameter
$\alpha_i$ equal to the minimum $\alpha$, and the other components are degenerate
without stochastic driving terms. 

We remark here that without assuming irreducibility
and aperiodicity, establishing subgeometric ergodicity in
the case $\varGamma v=0$ is difficult.
Consider the following example.
Let $n=1$, $\sigma(x)\equiv0$, $L(t)\equiv0$ for $t\ge0$, and 
\begin{equation*}
	\Bar b(x) = \begin{cases}
		-1 \,, & x\ge 0\,,\\[2pt]
		-1 - x\,, &x\le0\,.
	\end{cases} 
\end{equation*}
Clearly, $\Bar b(x)$ satisfies all the assumptions in
\cite{Arapostathis-Pang-Sandric-2019},
and
\begin{equation*}
	X^x(t)\,=\,x+\int_0^t\Bar b\bigl(X^x(s)\bigr)\,\D s,\qquad t\ge0,\quad x\in\RR\,.
\end{equation*}
A straightforward calculation shows that
\begin{equation*}
	X^x(t) \,=\, \begin{cases}
		\begin{cases}
			x-t \,, & 0\le t\le x \\[2pt]
			\E^{x-t}-1\,, &t\ge x\,,\\
		\end{cases}
		\,, & x\ge 0\,,\\[15pt]
		-1 +\E^{-t}+x\,\E^{-t}\,, &x\le0\,.
	\end{cases}
\end{equation*}
Let
\begin{equation*}
	\mathsf{d}(x,y)\,\df\,\frac{|x-y|}{1+|x-y|}\,,\qquad x,y\in\RR\,.
\end{equation*}
Then it is easy to see that the conditions (1)--(3) in
\cite[Theorem~2.4]{Butkovsky-14} hold.
However, condition (4) does not hold.
Namely, for arbitrary $t_0>0$ let $x,y>t_0$.
Then, $\mathsf{d}\bigl(X^x(t),X^y(t)\bigr)=\mathsf{d}(x,y)$
for all $t_0\le t\le x\wedge y$.

Let us mention that ergodic properties of piecewise Ornstein-Uhlenbeck processes
with jumps in the total variation norm have been considered in
\cite{Arapostathis-Pang-Sandric-2019,DG13,RZ11}.

\subsection{Piecewise Ornstein-Uhlenbeck processes with jumps under
	stationary Markov controls}

In \cref{S3.3} we consider a model with a constant control,
i.e. with the vector
$v\in\Delta\df\{u\in\RR^n\colon u\ge0,\ \langle e,u\rangle=1\}$
being constant and fixed.
If the scheduling policy (control) is a function of the state of the system,
then $v(x)$ in the limiting SDE \cref{PL-SDE} is, in general, a Borel measurable map
from $\RR^n$ to $\Delta$.
We call such a $v(x)$ a stationary Markov control and denote the set of such controls
by $\Usm$.
If $L_t\equiv0$ for $t\ge0$, or it is a compound Poisson process, it follows from the
results in \cite{Gyongy-96} that, under any $v\in\Usm$, \cref{PL-SDE} admits a unique
conservative strong solution which is a strong Markov process with c\`{a}dl\`{a}g
sample paths.
In the general case, we consider the subclass of stationary Markov controls for which
\begin{equation*}
	\Bar b_v(x) \,=\, l-M\bigl(x-\langle e,x\rangle^+v(x)\bigr)
	-\langle e,x\rangle^+\varGamma v(x)\,,
\end{equation*}
is locally Lipschitz continuous.
We let $\tUsm$ denote the class of such controls.
Clearly, for any $v\in\tUsm$, the drift $\Bar b_v(x)$ has at most linear growth.
Other parameters are as in \cref{S3.3}.
Again, the SDE of the form \cref{PL-SDE}, with $\Bar b(x)$ replaced by $\Bar b_v(x)$,
admits a unique conservative strong solution $\process{X}$ which is a strong Markov
process with c\`{a}dl\`{a}g sample paths.
Also, it is an It\^{o} process satisfying \hyperlink{MP}{\ttup{MP}} with
$b(x)=b_L+\Bar b_v(x)$, $a(x)=\sigma(x)\sigma(x)'$, and $\upnu(x,\D y)=\upnu_L(\D y)$.

Recently, in \cite{AHPS19-IMA-long} the authors have studied ergodic properties
with respect to the total variation norm of this model with $\process{L}$ being either
(or a combination of) a rotationally invariant $\alpha$-stable L\'evy process,
an anisotropic L\'evy process consisting of independent one-dimensional symmetric
$\alpha$-stable components, or a compound Poisson process.
Observe that in this situation we cannot follow the procedure from the constant
control case.
Namely, the matrices $Q\in\cM_+$ used in constructing the appropriate Lyapunov
functions $\Lyap(x)$ depend on $v$. 

\begin{proposition}\label{P3.6}
	Grant the assumptions of \cite[Theorem~3.1]{Arapostathis-Pang-Sandric-2019},
	and suppose that
		$M =\diag(m_1,\dotsc,m_n)$,
	with $m_i>0$ for $i=1,\dots,n.$
	\begin{enumerate}
		\item[\ttup{i}]
		Assume that the diagonal components of $\varGamma$ are strictly positive,
		$a(x)$ satisfies \cref{EP3.4A}, 
		and $\process{L}$ is either a rotationally invariant $\alpha$-stable L\'evy process,
		an anisotropic L\'evy process consisting of independent one-dimensional symmetric
		$\alpha$-stable components (in both cases we assume that $\alpha\in(1,2)$),
		or a compound Poisson process satisfying $1\in\Theta_\upnu$.
		We allow $\process{L}$ to have a drift.
		Then, for any $v\in\tUsm$ and $\theta\in[1,\theta_\upnu]\cap\Theta_\upnu$,
		the assertions of \cref{T1.1}\, \ttup{iii} hold true with $\eta=\theta$,
		and $\Lyap(x)=\bigl(\Bar\Lyap(x)\bigr)^\theta+1$, where
		$\Bar\Lyap\in C^2(\RR^n)$ 
		\textup{(}given explicitly in \cite[Definition 1]{AHPS19-IMA-long}\textup{)}
		is bounded from below away from zero, is Lipschitz continuous, and
		satisfies
		\begin{equation*}0\,<\,\liminf_{\abs{x}\to\infty}\frac{\Bar\Lyap(x)}{\abs{x}}\,\le\,
		\limsup_{\abs{x}\to\infty}\frac{\Bar\Lyap(x)}{\abs{x}}\,<\,\infty\,.\end{equation*}
		
		\item [\ttup{ii}]
		Assume $\bigl\langle e,M^{-1}\Tilde l\bigr\rangle<0$,
		where $\Tilde l$ is given in \cref{L},
		$a(x)$ satisfies \cref{A}, 
		and $\process{L}$ is a pure-jump L\'evy process (possibly with drift)
		satisfying $2\in\Theta_\upnu$.
		Then, for any $v\in\tUsm$ and $\theta\in[2,\theta_\upnu]\cap\Theta_\upnu$,
		the assertions of \cref{T1.1}\,\ttup i and \ttup{ii} hold true with
		$\phi(t)= t^{\nicefrac{(\theta-1)}{\theta}}$, $\eta=\theta-1$, and $\Lyap(x)$
		as in \ttup i.
		
		\item [\ttup{iii}] 
		In addition to the assumptions in \ttup{ii} assume that
		$\Tilde\theta_\upnu=\theta_\upnu\in(2,\infty)$, where $\Tilde\theta_\upnu$ is given
		in \cref{THETA}. 
		Then, due to \ttup{ii}, for any $v\in\tUsm$, $\process{X}$ admits a unique invariant
		$\uppi_v\in\cP_{\theta_\upnu-1-\iota}(\RR^n)$ for $\iota\in(0,\theta_\upnu-1)$.
		Next, fix $\rho\in(0,(\theta_\upnu-2)\wedge1)$ and $\varepsilon\in(\rho,1)$.
		Then, for any $v\in \tUsm$ such that $\varGamma v(x)=0$ a.e.,
		$p\in[1,\theta_\upnu-\rho-1]$ and $\iota\in(0,1-\varepsilon)$,
		there exist $\Bar c>0$ and a diverging increasing sequence
		$\{t_n\}_{n\in\NN}\subset[0,\infty)$, depending on these parameters,
		such that \cref{ET1.2A} holds for the corresponding $\uppi_v(\D x)$ with
		$\theta =\theta_\upnu-\rho$, $\vartheta=\theta-1$, and $\Lyap(x)$ as above.
	\end{enumerate}
\end{proposition}

\begin{proof}
	\begin{itemize}
		\item [\ttup {i}]
		Observe first that in the case when $\process{L}$ is a rotationally invariant
		$\alpha$-stable L\'evy process or an anisotropic L\'evy process consisting of
		independent one-dimensional symmetric $\alpha$-stable components,
		$\Theta_\upnu=[0,\alpha)$.
		In \cite[Theorem~3 and the discussion after Theorem~5]{AHPS19-IMA-long} 
		it has been shown that for any $v\in\tUsm$ and
		$\theta\in[1,\theta_\upnu]\cap\Theta_\upnu$ there exist
		$\Bar c=\Bar c(\theta,v)>0$ and $\Tilde c=\Tilde c(\theta,v)>0$,
		such that
		\begin{equation*}\mathcal{L} \bigl(\Bar\Lyap^\theta\bigr)(x)\,\le\,
		\Bar c-\Tilde c\, \bigl(\Bar\Lyap(x)\bigr)^{\theta}\qquad \forall\,x\in\RR^n\,.\end{equation*} 
		It is easy to see that the above relation implies that there exist $r>0$,
		$\Hat c>0$, and $\Breve{c}>0$, such that
		\begin{equation*}\mathcal{L} \Lyap(x)\,\le\, \Hat c\,\Ind_{\Bar{\sB}_r}(x)-\Breve c\, \Lyap(x)
		\qquad \forall\,x\in\RR^n\,.\end{equation*}
		The assertion then follows from \cref{T1.1}\,\ttup{iii}, together with
		\cite[Proposition 4.3]{Albeverio-Brzezniak-Wu-2010},
		\cite[Theorem~3.4]{Douc-Fort-Guilin-2009},
		and \cite[Theorems 5.1 and 7.1]{Tweedie-1994}. 
		
		\item [\ttup {ii}]
		In Theorem~5 and the discussion following the proof of this theorem in
		\cite{AHPS19-IMA-long} 
		it has been shown that for any $v\in\tUsm$ and
		$\theta\in(1,\theta_\upnu]\cap\Theta_\upnu$ there exist $r=r(\theta,v)>0$,
		$\Bar c=\Bar c(\theta,v)>0$, and $\Tilde c=\Tilde c(\theta,v)>0$,
		such that
		\begin{equation*}\mathcal{L} \bigl(\Bar\Lyap^\theta\bigr)(x)\,\le\,
		\Bar c\,\Ind_{\overline\sB_r}-\Tilde c\, \bigl(\Bar\Lyap(x)\bigr)^{\theta-1}
		\qquad \forall\,x\in\RR^n\,.\end{equation*}
		It is easy to see that the above relation implies that there exist
		$\Hat r>0$, $\check c>0$, and $\Breve{c}>0$, such that
		\begin{equation*}\mathcal{L} \Lyap(x)\,\le\, \check c\,\Ind_{\Bar{\sB}_{\Hat r}}(x)
		-\Breve c\, \bigl(\Lyap(x)\bigr)^{\nicefrac{(\theta-1)}{\theta}}
		\qquad \forall\,x\in\RR^n\,,\end{equation*}
		with $\Lyap(x)$ given as above.
		The assertion now follows from \cref{T1.1}\,\ttup{i} and \ttup{ii}, together with
		the results from \cite{Albeverio-Brzezniak-Wu-2010,Douc-Fort-Guilin-2009,Tweedie-1994}
		cited in part \ttup{i}. 
		
		\item [\ttup{iii}]
		Clearly, $\vartheta+\varepsilon>\theta_\upnu-1$.
		Thus, according to \cite[Lemma~5.7\,\ttup{b}]{Arapostathis-Pang-Sandric-2019},
		\begin{equation*}\int_{\RR^n}\bigl(\langle e,M^{-1}x\rangle^+\bigr)^{\vartheta+\varepsilon}\,
		\uppi_v(\D x)\,=\,\infty\,.\end{equation*}
		The assertion now follows from (the proof of) \ttup{ii}
		(together with the results from
		\cite{Albeverio-Brzezniak-Wu-2010,Douc-Fort-Guilin-2009,Tweedie-1994}
		cited in part \ttup{i}),
		and \cref{T1.2} by setting $L(x)=\Bar\Lyap(x)$ and
		$\phi(t)= t^{\nicefrac{(\theta-1)}{\theta}}$.\qedhere
	\end{itemize}
\end{proof}

As discussed in \cref{S3.3}, the hypothesis that $\Tilde\theta_\upnu=\theta_\upnu$ is true if $\process{L}$
is a compound Poisson process (possibly with drift) with L\'evy measure
$\upnu_L(\D y)$ supported on a half-line of the form
$\{tw\colon t\in[0,\infty)\}$ with $\langle e,M^{-1}w\rangle>0$.

Ergodic properties in the $f$-norm
of  piecewise Ornstein-Uhlenbeck processes with jumps under
stationary Markov controls
have been considered in \cite{arapostathis2018uniform,AHPS19-IMA-long}.

\subsection{ State-space models}

Let $F\colon\RR^n\to\RR^n$ be continuous, and such that $\abs{F(x)}\le c\abs{x}$
for some $c>0$ and all $x\in\RR^n.$
Further, let $X(0)$ be an $\RR^n$-valued random variable, and let $\{W(k)\}_{k\in\NN}$
be a sequence of i.i.d. $\RR^n$-valued random variables independent of $X(0)$.
Assume that the common distribution of $\{W(k)\}_{k\in\NN}$ has a nontrivial
absolutely continuous component which is bounded away from zero in a neighborhood
of the origin.
Then the Markov process defined by
\begin{equation*}X(k+1)\,=\, F\bigl(X(k)\bigr)+W(k+1)\,, \qquad k\ge0\,,\end{equation*}
is irreducible, aperiodic, and all compact sets are petite
(see \cite[Proposition 5.2]{TT-94}).
Further, assume that there exist constants $l\in\NN$, $l\ge2$, $\varepsilon\in(0,1)$,
and $\Bar c,r>0$, such that
\begin{equation*}\Exp\bigl[\abs{W_1}^l\bigr]\,<\,\infty\,, \qquad \text{and} \qquad \abs{F(x)}
\,\le\,c\abs{x}-\Bar c\abs{x}^{1-\varepsilon}\quad \forall\,x\in\sB_r^c\,.\end{equation*} 

\begin{proposition} \label{P3.7}
	Under the above assumptions, the assertions of \cref{T1.1}\,\ttup i and \ttup{ii}
	hold with $\Lyap(x)= \abs{x}^l$, $\phi(t)= t^{\nicefrac{(l-1)}{l}}$,
	and $\eta= l-1$.
\end{proposition}

\begin{proof}
	In \cite[Proposition 5.2]{TT-94} it has been proved that the Foster-Lyapunov
	condition in \cref{ET1.1A} holds with $\Lyap(x)$ and $\phi(t)$ as above,
	and $C=\sB_{\Bar r}$ for some $\Bar r>0$.
	The result now follows from \cref{T1.1}\,\ttup i and \ttup{ii}.
\end{proof}

Ergodic properties of state-space models in the $f$-norm have been
studied in \cite{TT-94,Fort-03}.

\subsection{Backward recurrence time chain}

Let $\{p_i\}_{i\ge0}\subset(0,\infty)$ be such that $p_0=1$, $p_i<1$ for $i\in\NN$,
and $\prod_{j=0}^ip_j\to0$, as $i\to\infty$.
Let $\{X(k)\}_{k\ge0}$ be a Markov process on $\{0,1,\dotsc\}$ defined by the
transition kernel $p(i,i+1)=1-p(i,0)\df p_i$ for $i\ge0$.
The process $\{X(k)\}_{k\ge0}$ is irreducible and aperiodic, and it admits a
unique invariant $\uppi\in\cP(\{0,1,\dotsc\})$ if, and only, if
\begin{equation*}c\,\df\,\sum_{i=1}^\infty\prod_{j=1}^ip_j\,<\,\infty\,.\end{equation*}
In this case, $\uppi(0)=\uppi(1)=(2+c)^{-1}$,
and $\uppi(i)=(2+c)^{-1}\prod_{j=0}^{i-1}p_j$ for $i\ge2$.

\begin{proposition}\label{P3.8}
	\begin{enumerate}
		\item [\ttup i]
		If there exist $i_0\in\NN$ and $\alpha >1$, such that $p_i=\frac{1+\alpha}{i}$
		for $i\ge i_0$,
		then the assertions of \cref{T1.1}\,\ttup i and \ttup{ii} hold
		with
		\begin{equation*}
			\Lyap (i)= i^{\beta(1+\alpha)}+1\,,\quad
			\phi(t)= t^{1-\frac{1}{\beta(1+\alpha)}}\,,\quad\text{and}\quad
			\eta= \beta(1+\alpha)-1\quad\text{for\ \ }
			\beta\in [\nicefrac{2}{(1+\alpha)},1)\,.
		\end{equation*} 
		
		\item [\ttup{ii}]
		Under the assumptions in \ttup{i}, $\uppi\in\cP_{\alpha-\iota}(\{0,1,\dotsc\})$ for
		$\iota\in(0,\alpha)$.
		Next, fix $\rho\in(0,(\alpha-1)\wedge1)$ and $\varepsilon\in[\rho,1)$.
		Then, for every $p\in[1,\alpha-\rho]$ and $\iota\in(0,1-\varepsilon)$ there
		exist a positive constant $c$ and a diverging increasing sequence
		$\{t_n\}_{n\in\NN}\subset[0,\infty)$, depending on these parameters,
		such that \cref{ET1.2A} holds with $\Lyap(i)$ as above,
		$\theta= 1+\alpha-\rho$, and $\vartheta= \alpha-\rho$.
	\end{enumerate}
\end{proposition}

\begin{proof}
	\begin{itemize}
		\item [\ttup i]
		In \cite[Section~3]{DFMS-04} it has been shown that the Foster-Lyapunov condition
		in \cref{ET1.1A} holds with a Lyapunov function $\Bar\Lyap(i)$ which asymptotically
		behaves like $\Lyap(i)$, $\phi(t)$ as above, and $C$ being a finite set for any
		$\alpha>0$ and $\beta\in(0,1)$.
		Taking into account \cref{ET1.1B}, the assertion follows.
		
		\item[\ttup {ii}]
		From the assumptions on the sequence $\{p_i\}_{i\ge0}$ we see that
		$\lim_{i\to\infty}i^{1+\alpha}\,\uppi(i)>0$.
		Now, since $\vartheta+\varepsilon-1-\alpha\ge-1$, we have 
$\sum_{i=0}^\infty i^{\vartheta+\varepsilon}\,\uppi(i)\,=\,\infty$.
		The assertion now follows from \cref{T1.2} by taking $L(i)= i$.\qedhere
	\end{itemize}
\end{proof}

\section{Concluding Remarks}

We remark on some other approaches
in the study of exponential or subexponential ergodicity of Markov processes.
By analyzing polynomial moments of hitting times of compact sets directly,
polynomial ergodicity results
are established in \cite[Theorem~6]{Veretennikov-1997}
for a class of irreducible (with respect to the Lebesgue measure)
and aperiodic diffusion processes.
In a follow-up work, by using analogous techniques, the same author established
polynomial ergodicity of a class of diffusion processes without directly assuming
irreducibility and aperiodicity of the process, but employing instead a so-called
(local) Dobrushin condition (also known as Markov-Dobrushin condition)
\cite[Theorem~6]{Veretennikov-2000}.
This approach is based on a Foster-Lyapunov condition of the form \cref{ET1.1A},
and instead of  assuming irreducibility and aperiodicity of $\process{X}$, it is
assumed that (i) $\Lyap(x)$ has precompact sub-level sets, and (ii) for every
$\delta>0$ there exists $t_\delta\in\T\setminus\{0\}$ such that
\begin{equation*}
	\sup_{(x,y)\in \{(u,v)\colon\,\Lyap(u)+\Lyap(v)\le\delta\}}\,
	\bnorm{p(t_\delta,x,\D z)-p(t_\delta,y,\D z)}_{\mathrm{TV}} \,<\, 1\,,
\end{equation*}
(see \cite[Chapter 3]{Kulik-Book-2018}).
Observe that this condition actually means that for each $(x,y)$ satisfying
$\Lyap(x)+\Lyap(y)\le\delta$ the probability measures $p(t_\delta,x,\D z)$
and $p(t_\delta,y,\D z)$ are not mutually singular.
Intuitively, the Dobrushin condition encodes irreducibility and aperiodicity of
$\process{X}$, and petiteness of sub-level sets of $\Lyap(x)$. 
Based on these assumptions, and using an appropriate Markov coupling of $\process{M}$,
it follows that the $\Phi^{-1}$-modulated moment of the corresponding coupling
time is finite and controlled by $\Lyap(x)+\Lyap(y)$.
This then implies (sub)geometric ergodicity of $\process{X}$ in the total variation
norm (see \cite[Theorem~4.1]{Hairer-Lecture-notes-2016}
or \cite[Chapter 3]{Kulik-Book-2018}). 

We remark that irreducibility and aperiodicity (together with \cref{ET1.1A}) imply
that the Dobrushin condition holds on the Cartesian product of any petite set with itself.
Namely, according to \cite[Proposition 6.1]{Meyn-Tweedie-AdvAP-II-1993}, for any petite
set $C$ there exists $t_C\in\T\setminus\{0\}$ such that for the measure
$\upchi(\D t)$ (in the definition of petiteness) the Dirac measure in $t_C$ can be
taken (together with some non-trivial measure $\upnu_\upchi(\D x)$).
Thus, $p(t_C,x,B)\ge\upnu_\upchi(B)$ for any $x\in C$ and $B\in\mathfrak{B}(\X)$,
which implies that
\begin{equation}\label{DOB}
\sup_{(x,y)\in C\times C}\, \bnorm{p(t_C,x,\D z)-p(t_C,y,\D z)}_{{\mathrm{TV}}}
\,<\, 1\,.
\end{equation}
If, in addition, $\process{X}$ is
$C_b$-Feller (i.e. $x\mapsto \int_{ \X}f(y)\,p(t,x,\D y)$ is continuous
and bounded for any $t\in\T$ and any continuous and bounded function $f(x)$),
and the support of the corresponding irreducibility measure has nonempty interior,
then every compact set is petite (see \cite[Theorems 5.1 and 7.1]{Tweedie-1994})
and thus \cref{DOB} holds for any bounded set $C$.
This shows that,
at least in this particular situation, the approach based on the Dobrushin condition
is more general than the approach based on irreducibility and aperiodicity.
Situations where it has a clear advantage are discussed in
\cite{Kulik-2009,Abourashchi-Veretennikov-2010}.
In \cite{Kulik-2009}, the author considers
a Markov process obtained as a solution to a L\'evy-driven SDE with highly
irregular coefficients and noise term; while in
\cite{Abourashchi-Veretennikov-2010}, a diffusion process with highly irregular
(discontinuous) drift function and uniformly elliptic diffusion coefficient has been
considered.
In these concrete situations it is not clear whether one can obtain irreducibility
and aperiodicity of the processes, whereas the authors obtain \cref{DOB} for any
compact set $C$ (see \cite[Theorem~1.3]{Kulik-2009}
and \cite[Lemma~3]{Abourashchi-Veretennikov-2010}).
For additional results
on ergodic properties of Markov processes based on the Dobrushin condition
we refer the readers to \cite{Hairer-Lecture-notes-2016,Kulik-Book-2018}.

\section*{Acknowledgements}
We thank the anonymous referee for the helpful comments
that have led to significant improvements of the results in the article.
This research was supported in part by 
the Army Research Office through grant W911NF-17-1-001,
and in part by the National Science Foundation through grants DMS-1715210,
CMMI-1635410, and DMS-1715875.
Financial support through the Alexander von Humboldt Foundation (No. HRV 1151902 HFST-E) and the
Croatian Science Foundation under the project 8958
(for N. Sandri\'c) is gratefully acknowledged.


\def\cprime{$'$}

\end{document}